\documentclass[aap]{imsart}
\usepackage{amsthm,amsmath,amsfonts,amssymb}
\usepackage[numbers]{natbib}
\usepackage[colorlinks,citecolor=blue,urlcolor=blue]{hyperref}
\usepackage{graphicx}

\usepackage{ latexsym, amsthm}
\usepackage[mathscr]{eucal}
\usepackage{graphics}
\usepackage{color}
\usepackage{hyperref}
\usepackage{epsfig}
\usepackage{tikz}
\usepackage{mathtools}
\usetikzlibrary{shapes,arrows}
\usetikzlibrary{graphs,graphs.standard}
\usepackage{pgfplots}
\usepackage{dsfont}
\usepackage{subfig}
\usepackage[OT1]{fontenc}
\usepackage{amsthm,amsmath, dsfont}
\usepackage[colorlinks,citecolor=blue,urlcolor=blue]{hyperref}
\usepackage{ bm, dsfont, mathtools, soul, tikz}
\usepackage[shortlabels]{enumitem}

\startlocaldefs
\numberwithin{equation}{section}
\newtheorem{theorem}{Theorem}[section]
\newtheorem{lemma}[theorem]{Lemma}
\newtheorem{corollary}[theorem]{Corollary}
\newtheorem{definition}[theorem]{Definition}

\newtheorem{proposition}[theorem]{Proposition}
\newtheorem{assumption}[theorem]{Assumption}
\newtheorem{conjecture}[theorem]{Conjecture}

\theoremstyle{remark}
\newtheorem{remark}[theorem]{Remark}
\newcommand{\beginsec}{
	\setcounter{equation}{0}
}

\DeclareMathOperator*{\tildeinf}{\widetilde{\inf}}
\newcommand{\la}{\lambda}
\newcommand{\eps}{\varepsilon}
\newcommand{\ph}{\varphi}

\newcommand{\al}{\alpha}

\newcommand{\s}{\sigma}
\newcommand{\sig}{\sigma}
\newcommand{\del}{\delta}

\newcommand{\Gam}{\mathnormal{\Gamma}}
\newcommand{\Del}{\mathnormal{\Delta}}

\newcommand{\X}{\mathnormal{\Xi}}

\newcommand{\Sig}{\mathnormal{\Sigma}}
\newcommand{\Ups}{\mathnormal{\Upsilon}}
\newcommand{\Ph}{\mathnormal{\Phi}}

\newcommand{\Om}{\mathnormal{\Omega}}

\newcommand{\N}{{\mathbb N}}
\newcommand{\Q}{{\mathbb Q}}
\newcommand{\R}{{\mathbb R}}

\newcommand{\Z}{{\mathbb Z}}

\newcommand{\BH}{{\mathbb H}}

\newcommand{\E}{{\mathbb E}}

\newcommand{\PP}{{\mathbb P}}

\newcommand{\calA}{{\cal A}}
\newcommand{\calB}{{\cal B}}

\newcommand{\calE}{{\cal E}}
\newcommand{\calF}{{\cal F}}
\newcommand{\calG}{{\cal G}}
\newcommand{\calH}{{\cal H}}
\newcommand{\calI}{{\cal I}}
\newcommand{\calJ}{{\cal J}}
\newcommand{\calK}{{\cal K}}

\newcommand{\calM}{{\cal M}}
\newcommand{\calN}{{\cal N}}

\newcommand{\calS}{{\cal S}}
\newcommand{\calT}{{\cal T}}
\newcommand{\calU}{{\cal U}}

\newcommand{\frS}{\mathfrak{S}}
\newcommand{\frX}{\mathfrak{X}}

\newcommand{\lan}{\langle}
\newcommand{\ran}{\rangle}

\newcommand{\w}{\wedge}

\newcommand{\To}{\Rightarrow}

\newcommand{\iy}{\infty}

\newcommand{\be}{\begin{equation}}
\newcommand{\ee}{\end{equation}}

\newcommand{\noi}{\noindent}
\newcommand{\ds}{\displaystyle}
\newcommand{\one}{\mathds{1}}

\newcommand{\slp}{\calS_{\text{\rm LP}}}
\newcommand{\LE}{\preccurlyeq}
\newcommand{\les}{\leqslant}
\newcommand{\ges}{\geqslant}

\newcommand{\indep}{\perp\hspace{-0.6em}\perp}
\newcommand{\sgn}{{\rm sgn}}

\newcommand{\skp}{\vspace{1em}}
\endlocaldefs

\begin{document}

\begin{frontmatter}
\title{Parallel server systems under an extended heavy traffic condition: A lower bound
}
\runtitle{PSS under an EHTC: a lower bound}

\begin{aug}
\author[A]{\fnms{Rami} \snm{Atar}\ead[label=e1, ,mark]{rami@technion.ac.il}},
\author[A]{\fnms{Eyal} \snm{Castiel}\ead[label=e2,mark]{eyal.ca@campus.technion.ac.il}}
\and
\author[B]{\fnms{Martin I.} \snm{Reiman}\ead[label=e3,mark]{martyreiman@gmail.com}}
\address[A]{Viterbi Faculty of Electrical and Computer Engineering, Technion, Haifa, Israel, \printead{e1,e2}}

\address[B]{Department of Industrial Engineering and Operations Research, Columbia University, New York City, NY, \printead{e3}}
\end{aug}

\begin{abstract}
The standard setting for studying parallel server systems (PSS) at the
diffusion scale is based on the heavy traffic condition (HTC), which
assumes that the underlying static allocation linear program (LP) is critical and has a unique solution. This solution determines the graph of basic activities,
which identifies the set of activities (i.e., class--server pairs) that are operational.
In this paper we explore the extended HTC, where the LP is merely assumed to be critical.
Because multiple solutions are allowed, multiple
sets of operational activities, referred to as modes, are available.
Formally, the scaling limit for the control problem
associated with the model is given by a
so-called workload control problem (WCP)
in which a cost associated
with a diffusion process is to be minimized by dynamically
switching between these modes.
Our main result is that the WCP's value constitutes an
asymptotic lower bound on the cost associated with the PSS model.
\end{abstract}

\begin{keyword}[class=MSC]
	\kwd[Primary: 60K25; 90B22; 68M20; 49K45]{}
	\kwd[; Secondary:  60F17; 93E20; 49K15]{}
\end{keyword}

\begin{keyword}
	\kwd{parallel server systems}
	\kwd{heavy traffic}
	\kwd{diffusion limits}
	\kwd{Brownian control problem}
	\kwd{Hamilton-Jacobi-Bellman equation}
	\kwd{strict complementary slackness}
\end{keyword}

\end{frontmatter}

\setcounter{tocdepth}{1}
\tableofcontents


\section{Introduction}\label{sec1}

\subsection{Background and motivation}

Parallel server systems (PSSs) constitute a class of queueing control problems that have received a lot of attention since their introduction
in  \cite{harlop}.
The overwhelming majority of this effort has been on heavy traffic limits and associated asymptotic optimality (AO).
Almost all of this work has been carried out under the assumption that a particular linear program
(LP), called the `static allocation problem'
in  \cite{harlop} has a unique solution. 
(This LP was called the static planning problem in \cite{har2000}, which deals with a broader set of problems.)
In particular, the notion of heavy traffic for PSS  defined in \cite{harlop} requires that the static allocation problem has a unique solution. 
But it is not difficult to construct a PSS that is intuitively in `heavy traffic' and yet,
because its associated static allocation problem does not have a unique solution,
does not satisfy the the heavy traffic definition of \cite{harlop}. 
Thus what is termed heavy traffic in \cite{harlop} is really, in a sense, \emph{restricted} heavy traffic.
In this paper we remove the requirement that the solution of the static planning problem is unique.
Because the term heavy traffic in the context of PSS has already been defined in \cite{harlop} , we use the term extended heavy traffic.

A few further comments on the uniqueness assumption are in order.
First, this assumption carries with it the benefit of substantial simplification of the problem, described further below.
As a consequence, as also described further below, there has been great progress on this problem in \cite{bw1},\cite{bw2}, \cite{harlop}, 
and \cite{man-sto} under the uniqueness assumption.
On the other hand, there is no strong reason
to believe that the uniqueness assumption holds in typical problems of practical interest.
For example,  it fails to hold under natural properties arising in the design of service systems, such as the decomposability condition described at the end of this
subsection. Thus it seems reasonable to suspect that it often does not hold.
While the great progress just cited justifies the insightful restriction of uniqueness as a temporary expediency, 
we feel that it is time to move beyond it.
The purpose of this paper is to begin the exploration of what happens if the uniqueness
assumption does not hold. This paper provides an asymptotic lower bound for a general PSS. 
A companion paper  \cite{ourupperbound} provides asymptotic optimality results for the simplest non-trivial PSS where uniqueness fails to hold.

The specific system that we consider, described in more detail in \S 2, has $I$ classes of customers and $K$ servers. A class-server pair $j=(i,k)$ is said to be an activity
if server $k$ is capable of serving class $i$. The number of activities is denoted by $J$.
The service rate $\mu_{j}$ depends on the activity, namely the class-server combination.
As is the case in all PSS models, customers arrive at a class based queue
where they await service, receive a single service, and leave the system. 
The controls available here are both routing (which server serves each customer) and sequencing (in what order  they are served). Our objective is to minimize the expected infinite horizon total discounted cost, where the cost rate is a linear function of the queue length.

Despite the simple specification of this control problem, determining an optimal control is typically very challenging.
This motivated  Harrison and Lopez \cite{harlop} to look for control policies that are AO in heavy traffic. When the control problem involves only sequencing - the order in which customers are served, the traffic intensity can be 
obtained immediately from the arrival and service rates.
On the other hand, when the control problem involves routing, as it does in PSSs, the traffic intensity may depend on the control policy that is used.
As a consequence, it is not immediately obvious how to define heavy traffic.
This is the context in which Harrison  \cite{har2000} introduced the static planning problem,
noting the connection to earlier work of \cite{laws92}, which introduced a multicommodity network flow problem to define heavy traffic for dynamic routing problems.

Harrison \cite{har1988} pioneered the notion of using Brownian Control Problems (BCPs) to obtain
AO control policies for  queueing network control problems.
As described in \cite{harlop}, the solution of such a problem can be broken into 4 steps: 
(1) Derive a BCP that plausibly approximates the original control problem; 
(2) Solve the BCP; (3) Translate the solution of the BCP into a control policy for the original problem; and (4) Prove that this policy is AO.

In general, the BCP that arises from 
queueing control problems, including 
PSSs, is multidimensional, and consequently difficult to solve.
It was shown in \cite{Har-Van} how state space collapse can lead to an `equivalent workload formulation' of lower dimension.
In the context of PSSs,
a dramatically simplifying condition  known as `complete resource pooling' (CRP), under which, in heavy traffic,
the servers can effectively be `pooled' together into a single `super server', was introduced in \cite{harlop}.
They showed that the CRP condition yields a BCP whose equivalent workload formulation  is one dimensional.  

A key assumption in both    \cite{har2000}, \cite{harlop} 
is that the static planning LP has a unique solution. This assumption, combined with CRP, gives rise to a  substantial simplification: 
It was shown in  \cite{harlop}  that
the equivalent workload formulation of the BCP has a simple pathwise solution, which minimizes the instantaneous cost rate at every point in time with probability one, completing steps (1) and (2) under the uniqueness assumption.  
Almost all subsequent work that we are aware of has involved this assumption. The exceptions (that we are aware of) are described below. 
It was understood and acknowledged that relaxing the uniqueness assumption would lead to additional difficulties. For example,
in  \cite{harlop}  Harrison and Lopez note that  `multiple optima lead to Brownian approximations of a  more complicated form than we are prepared to treat'.

Steps (3) and (4) above were first completed in the context of a simple special case known as an `N' network.
This has 2 customer classes, 2 servers, and 3 activities. (Server 1 can only serve class 1.)
The `N' structure assures that the static allocation LP can have at most 1 solution.
An `N' network with specific distributional assumptions and parameters was considered in \cite{har98}, where a discrete review 
policy,  obtained via the `BIGSTEP' method, was introduced and shown to be AO.
This same N network (but with  more general distributions and parameter values) was considered by Bell and Williams in \cite{bw1}. The parameter values are chosen to yield heavy traffic.  Under the objective of minimizing the expected infinite horizon total discounted cost they prove AO of a  threshold policy. 

Bell and Williams \cite{bw2} treat the general PSS under the assumptions that the static allocation LP has a unique solution, and that the CRP condition holds. Under those conditions AO of an intricately designed priority-threshold policy is proved.

PSSs were considered by Mandelbaum and Stolyar \cite{man-sto} under a cost function that is increasing and strictly convex. In this context they showed that a simple generalized $c\mu$-rule is AO. They also effectively assumed uniqueness of the solution to the static allocation LP, along with CRP. They acknowledge difficulty if uniqueness does not hold, and mention that, when uniqueness fails to hold due to having groups of identical servers, by  viewing the groups as individual servers their results continue to hold. 

Stolyar  \cite{sto05}  introduces a condition that he calls First-Order-CRP (FO-CRP), which relaxes the requirement that the static allocation LP has a unique solution. However, the main heavy traffic AO results are proved under the more 
restrictive CRP condition, which requires uniqueness.

Of course, PSSs that do not satisfy CRP, as well as more complex stochastic networks remain of great interest.
Some progress on PSSs that do not satisfy CRP was obtained by Pesic and Williams  \cite{pes16}, who
prove some structural results along with an optimal control for the BCP under certain conditions.
It is conjectured there that a particular continuous review threshold policy is AO
under those conditions.

In the context  of more complex stochastic networks,  
Budhiraja and Ghosh \cite{BG2} show that the solution of the BCP provides an asymptotic lower bound for a class of controlled queueing networks, and in \cite{BG2012} they show that the value function converges to this lower bound.
It is worth noting that \cite{pes16}, \cite{BG2} and \cite{BG2012}  all assume that the static planning LP has a unique solution.

As mentioned above, there is no strong reason to believe that the static allocation LP will satisfy the uniqueness assumption  in typical problems of practical interest.
In fact, there is good reason to believe that it will often not hold. An interesting and natural case is where the service rates satisfy a decomposition property.
Let $\mu_{ik}$ denote the rate at which class $i$ jobs are processed
by server $k$, and assume that every class-server pair is an activity.
As we show, assuming that these rates (to first order) are decomposable
as $\mu_{ik}=\al_i\beta_k$ always gives rise to non-uniqueness.
In the case $I=K=2$ we show that such a decomposition is
also necessary for non-uniqueness.
Service rates decompose this way when
the mean size of a job is characteristic to the class (and then
$\al_i$ is the reciprocal mean), and each server has its own processing speed
(here given by $\beta_k$).
Because all earlier results on the PSS in heavy traffic were obtained
under the HTC, this class of service rates has been excluded.
This decomposition does not have to hold exactly, just to first order, allowing the second order data, which determine drift and variance, to be more general.
In addition, this decomposition need not hold for all classes and servers. Non-uniqueness can arise when there are subsets of classes and/or servers of size at least two for which the decomposition holds,  even if not all class-server pairs form activities.
Examples (D) and  (E) in  \S \ref{sec:ex} provides   examples of this.

\subsection{Results}

A brief informal
description of the role played by the uniqueness assumption
in earlier work will help explain the main issues that must be dealt with
in its absence.
Whereas the aforementioned BCP describes the diffusion scale
deviations of all relevant stochastic processes
(such as queue lengths, cumulative busyness time on
each activity, etc.)\ about a central model,
the static allocation problem is concerned with the fluid scale behavior
and determines the central model itself.
Specifically, this LP describes the mean fraction of time
devoted by each server to the different classes.
When the LP has a unique solution,
any sequence of policies along which the diffusion-scaled cost
remains bounded as the scaling parameter grows
must at the very least `adhere' to this solution at the fluid scale.
That is, such a sequence of policies should have the property that
the diffusion scale deviations of the relevant processes about the central model
are at most of order 1, so they can relate to the processes comprising the BCP.
As a consequence, such a sequence
must have the much weaker, fluid scale property,
that the fractions of cumulative busyness for each activity
converge to this unique solution at the scaling limit.
The BCP is set up in a way that reflects behavior
according to this unique solution.

The activities to which the LP solution
allocates a strictly positive (respectively, zero) fraction of effort are said to be basic
(respectively, non-basic).
In the work \cite{bw2} (which, like this paper, considers linear holding costs)
the non-basic activities are not used by the prescribed policies
even for diffusion-scale corrections.
In this sense, the solution to the LP determines not only the fluid limit
of the model but also which of the activities are operational.

When the LP has multiple solutions, many of the aspects
described above break down.
A BCP can still be written down, describing
deviations about a fluid-scale allocation in the space
of solutions, but the fluid-scale allocation now becomes a stochastic process in
the space of solutions that is  controlled by the decision maker.
Thus the model has a non-deterministic fluid limit.
The set of basic, hence operational, activities has to be selected dynamically.
The introduction of this dynamics influences all stages of the analysis:
The form of the BCP and the method of solving it,
techniques for proving weak convergence required for establishing the lower bound
(and, in the paper \cite{ourupperbound}, the design of policies and the proof of their AO).

The class of PSS treated in this paper is quite general, except
for the restriction to systems where the workload process is one-dimensional.
In that we have roughly followed \cite{harlop} and \cite{bw2}, which restrict their attention to systems that satisfy CRP.
As mentioned previously, under CRP the servers can be pooled into a single super server, yielding a one-dimensional workload process.
It was shown in Proposition 4 of \cite{harlop} that under CRP (which assumed uniqueness of the LP solution) the equivalent workload formulation is one-dimensional.
We do not prove an analog of Proposition 4 of \cite{harlop} under non-uniqueness. On the other hand, we show that
the condition for reduction to a one-dimensional problem associated with
the workload process is the same as with a unique solution, namely that the dual of the static allocation LP has a unique solution. 
Because the definition of CRP has been tightly tied to uniqueness we do not use the term in relation to our results.

In what follows we summarize our results.

{\it Workload control problem (WCP) and Hamilton-Jacobi-Bellman (HJB) equation.}
Under multiplicity,
the set of LP solutions has $2\le M<\infty$ extreme points that we call control modes,
or simply modes.
Each mode corresponds to a different allocation and a different
set of operational activities. 
As mentioned above, one of the dynamically controlled BCP ingredients
is the fluid-scale limit of the allocation process. This process is in general random. 

We derive the BCP and then introduce the WCP,
which is the equivalent workload formulation of the BCP in our case.
The WCP is formulated in terms of a one-dimensional
stochastic differential equation (SDE)
with controlled drift and diffusion coefficients, and reflection at zero.
It is dramatically different 
from the WCP obtained under the uniqueness assumption,
in which the coefficients are not controlled (and are constant), and
the solution is a
reflected Brownian motion (BM).
It is a standard fact that the solution to a control problem of this kind
can be characterized in terms of viscosity solutions of an HJB equation.
The general theory allows one to express an optimal (or nearly optimal)
control as a feedback from the diffusion process (where the feedback function
can be specified in terms of the solution to the HJB equation).
In our case we can prove, based on PDE theory on fully nonlinear uniformly elliptic equations,
that the HJB equation has a unique
classical solution. Moreover, the solution to the WCP is given by
a diffusion with discontinuous drift and diffusion coefficients.
Examples of queueing models and queueing control models in heavy traffic which lead to
a diffusion process with a discontinuous diffusion coefficient have appeared in the literature
\cite{kry02}, \cite{rath}. To the best of our knowledge this is the first time this occurs
in the context of optimal/asymptotically optimal control of a queueing system.

{\it Asymptotic lower bound.}
Our main result states that the solution of the WCP provides an asymptotic lower bound on achievable cost.
The lack of a pathwise solution to the WCP, and the related representation of the limit process introduces technical difficulties.
To prove the result one must show that scaling limits
of the queueing model processes form an admissible control system
for the WCP. The specific tuple that forms such a control system
includes a filtration and a standard BM (SBM) on this filtration,
which serves as the driving BM in the aforementioned SDE.
The construction of this BM can be achieved via a transformation of certain
processes associated with the scaling limits of
the arrival processes and random time changes of
potential service processes of the $J$ activities, once it is established that these
processes are all martingales on a common filtration.
The random time changes appear
for exactly the reason that the modes vary stochastically.
Whereas in the single mode case the fluid limits of cumulative busyness processes
are deterministic and the time changed processes are easily shown to be martingales,
this conclusion is not immediate in the multiple mode case.
Our approach to this difficulty is based on multi-parameter filtrations
and a result due to Kurtz \cite{kurtz80} on multi-parameter time change.
One of the main steps toward using this result is to show
that the prelimit cumulative busyness processes define multi-parameter
stopping times on a filtration generated by the model's primitives.

Our results do not fit neatly into  the $4$ steps of \cite{harlop}. 
This paper covers step (1), but also parts of steps (2) and (4).
It partly covers step (2) by characterizing the solution of the WCP in terms of an HJB equation. It also covers part of step (4) by providing a target that, if achieved, yields AO.
The remainders of steps (2) and (4), along with step (3), are carried out in \cite{ourupperbound} for the simplest PSS, consisting of 2 classes, 2 servers
and 4 activities ($I=K=2$, $J=4$).

\subsection{Organization of the paper}
In \S \ref{sec20}, we present the notation and model in more detail.
We define the extended HTC in \S\ref{sec21}, and state
lemmas regarding the static allocation LP and its dual.
We introduce the WCP and state our main result in \S \ref{sec:wcpp}.
A discussion on degeneracy of the LP solution appears in \S
\ref{sec:degen}.
Several examples of PSS with multiple LP solutions are provided in \S \ref{sec:ex}.
In \S \ref{sec3}, we prove the lemmas related to the LP.
\S \ref{sec4} treats the BCP, WCP and HJB equation.
In \S \ref{sec41}, the BCP and WCP are derived.
\S \ref{sec42} provides a candidate solution to the WCP based
on the HJB equation. \S \ref{sec43} proves that a classical
solution to the HJB equation uniquely exists and establishes
the optimality of the WCP's candidate solution.
It also addresses a special case where the WCP's solution is to always use one
particular mode.
Finally, we prove the main result in \S \ref{sec5}.
The main lemmas on which this proof is based
appear in \S \ref{sec51},
and are proved in \S \ref{sec52}.
For ease of reference several results related to linear programming that we use are collected together in the Appendix.

\subsection{Notation and terminology}

The following notation will be used throughout.
$\N$, $\R$ and $\R_+$ are the sets of natural, real and, respectively,
non-negative real numbers. The symbol $\mathbf{i}$ denotes $\sqrt{-1}$.
For $a, b \in \R$, $a \vee b$ denotes the maximum of $a$ and $b$,
$a \wedge b$  the minimum of $a$ and $b$,  $a^+=a \vee 0$, and $a^-=-(a \wedge 0)$.
For any $P\in \N$, $0_P$ (resp. $1_P$) denote column vectors in $\R^P$ whose all components are 0 (resp. 1).
For a set $A$, $\one_A$ denotes the indicator function of $A$.
For $a,b\in\R^P$, write $a\les b$ (equivalently $b\ges a$) for the componentwise
inequalities. For a set $A\subset\R^P$, $\text{\rm ch}(A)$ denotes the closed
convex hull of $A$.
For real-valued functions and processes, the notation $X(t)$ is
used interchangeably with $X_t$.
Given a Polish space $E$, denote by $C_E[0,\iy)$ and $D_E[0,\infty)$
the spaces of $E$-valued, continuous and, respectively, c\`{a}dl\`{a}g functions on $[0,\infty)$.
Equip the former with the topology of convergence u.o.c.\ (uniformly over compacts) and the latter with the $J_1$ topology. For an element of $C_E[0,\iy)$ (or $D_E[0,\infty)$), we use the notation $\|f\|_t=\sup_{s\in [0,t]} \lvert f_s\rvert$.
Denote by $C^+_\R[0,\iy)$ (respectively, $D^+_\R[0,\infty)$) the subset of $C_\R[0,\iy)$
(respectively, $D_\R[0,\iy)$) of non-negative, non-decreasing functions,
and by $C^{0,+}_\R[0,\iy)$the set of functions in $C^+_\R[0,\iy)$
that are null at zero.
Write $X_n\To X$ for convergence in law.
A tight sequence of processes with sample paths in $D_E[0,\iy)$ is said to be $C$-tight
if it is tight and the limit of every
weakly convergent subsequence has sample paths in $C_E[0,\iy)$ a.s.
The letter $c$ denotes a positive deterministic
constant whose value may change from one
appearance to another.

In addition to abbreviations already introduced, we will use
RV for random variable(s); IID for independent identically distributed; LHS (resp.\ RHS) for left (resp.\ right) hand side.

\section{Model and main results}\label{sec2}
\beginsec

\subsection{Queueing model, scaling and queueing control problem}
\label{sec20}

The queueing model is described in terms of a sequence of systems, indexed by $n\in\N$,
all defined on one probability space $(\Om,\calF,\PP)$.
Let us first introduce the structure of the system.
Classes belong to the set $\mathcal{I}$ of cardinality $I$ and throughout, $i$ will be used as a generic element of $\calI$, indexing a class. Similarly, servers belong to the set $\mathcal{K}$ of cardinality $K$ with elements denoted $k$. Additionally, if a server $k$ is able to serve a particular class $i$, the pair $(i,k)$ is called an {\it activity}. The set of all activities is denoted by $\mathcal{J}\subset \mathcal{I}\times \mathcal{K}$
and its cardinality by $J$.
Elements of $\mathcal{J}$ are  denoted $j=(i,k)$. We will write $\calJ_i$ (resp. $\calJ^k$) for the set of all activities involving class $i$ (resp., server $k$).

For the $n$th system, the processes of interest are as follows.
First, the processes $A^n=(A^n_i)_{i\in \calI}$, $S^n=(S^n_{j})_{j\in \calJ}$
represent arrival and potential service counting processes.
That is, $A^n_i(t)$ is the number of arrivals of class $i$ jobs until time $t$, $i\in \calI$,
and $S^n_{j}(t)$ is the number of service completions on activity $j$, by the time the server has devoted $t$ units of time to this activity,
$j\in \calJ$. Next, $X^n=(X^n_i)_{i\in \calI}$, $I^n=(I^n_k)_{k\in \calK}$, $D^n=(D^n_{j})_{j\in\calJ}$ and $T^n=(T^n_{j})_{j\in\calJ}$
denote queue length, cumulative idleness, departure, and cumulative busyness
processes. In other words, $X^n_i(t)$ is the number of class $i$ customers in the system at time $t$,
$I^n_k(t)$ is the cumulative time server $k$ has been idle by time $t$, $D^n_{j}(t)$
is the number departures due to activity $j$, and $T^n_{j}(t)$
is the cumulative time devoted to activity $j$.
The process $T^n_{j}$ takes the form $T_{j}(t)=\int_0^t\X^n_{j}(s)ds$,
where $\X^n_{j}(t)$ is the fraction of effort devoted to activity $j$. In particular, $\sum_{j\in \calJ^k}\X^n_{j}(t)\le1$ for every $k$.
Thus $\X^n$ is referred to as the {\it allocation process}.

The probabilistic model for the primitive processes is as follows.
Arrival rates $\la^n_i$ and service rates
$\mu^n_{j}$ are given, satisfying, for some constants
$\la_i\in(0,\iy)$, $\hat\la_i\in\R$ for any $i \in \calI$, and $\mu_{j}\in(0,\iy)$,
$\hat\mu_{j}\in\R$  for any $ j\in \calJ$,
\begin{align*}
	&
	\hat\la^n_i:=n^{-1/2}(\la^n_i-n\la_i)\to\hat\la_i,\\
	&
	\hat\mu^n_{j}:=n^{-1/2}(\mu^n_{j}-n\mu_{j})\to\hat\mu_{j},
\end{align*}
as $n\to\iy$.
For each $i$ a renewal process $\check A_i$ is given,
with interarrival distribution that has mean 1 and squared
coefficient of variation $0<C^2_{A_i}<\iy$.
Similarly, for each $j$, a renewal process $\check S_{j}$
is given with mean 1 interarrival and squared coefficient of
variation $C^2_{S_{j}}<\iy$. It is assumed that
$A^n$ and $S^n$ are given by
\[
A^n_i(t)=\check A_i(\la^n_it),
\qquad
S^n_{j}(t)=\check S_{j}(\mu^n_{j}t).
\]
We call the parameters $\{\lambda_i\}_{i\in\calI}$, $\{\mu_j\}_{j\in\calJ}$ the {\it first order data}
and $\{\hat\la_i\}_{i\in\calI}$, $\{\hat\mu_j\}_{j\in\calJ}$, $\{C_{A_i}\}_{i\in\calI}$ and $\{C_{S_j}\}_{j\in\calJ}$
the {\it second order data}.
It is assumed that the $I+J$ processes $\check A_i$, $\check S_{j}$
are mutually independent, have strictly positive inter-arrival distributions
and right-continuous sample paths.
The (IID) interarrivals of $\check A_i$ and $\check S_{j}$ are denoted by
$\check{a}_{i}(p)$ and $\check{u}_{j}(p)$, respectively, and those of the accelerated
processes $A^n_i$ and $S^n_{j}$ are given by
\begin{equation}\label{e01}
	a^n_{i}(p)=\dfrac{1}{\lambda^n_{i}}\check{a}_{i}(p),
	\qquad
	u^n_{j}(p)=\dfrac{1}{\mu^n_{j}}\check{u}_{j}(p).
\end{equation}
The system is assumed to start empty, that is, $X^n(0)=0$ for all $n$.

The first order service rates are encoded in an $I\times J$ matrix $R$. The element $R_{ij}=\mu_{j}$ gives the rate at which activity $j$ processes class $i$ if $j\in \calJ_i$ and one sets $0$ otherwise. Additionally, we will need a $K\times J$ matrix $G$. Elements of $G$ are either $0$ or $1$, and are given by $G_{k,j}=\mathds{1}_{j\in \calJ^k}$. The matrix $G$ serves as interference constraint between the activities.

Simple relations between the processes introduced above
are expressed as follows,
\begin{equation}\label{40-}
	D^n_{j}(t)=S^n_{j}(T^n_{j}(t)),
\end{equation}
\begin{equation}\label{40}
	X^n_i(t)=A^n_i(t)-\sum_{j\in \calJ_i}D^n_{j}(t),
\end{equation}
\begin{equation}\label{41}
	I^n_{k}(t)=t-(GT^n(t))_k,
\end{equation}
\begin{equation}\label{41+}
	\text{$X^n_i$ is non-negative for all $i$, and $I^n_{k}$ is non-negative non-decreasing for all $k$.}
\end{equation}

The tuple $(\check A,\check S)$ is referred to as the {\it stochastic primitives}.
In our formulation we will consider $T^n$ as the control process, which is obviously
equivalent to having the allocation process $\X^n$ as the control.
In view of equations \eqref{40-}, \eqref{40}, \eqref{41},
given the stochastic primitives, the control uniquely determines the processes $D^n$, $X^n$, $I^n$.
Let an additional process be defined on the probability space
denoted by $\Ups=(\Ups(k),k\in\N)$,
taking values in a Polish space $\calS_\text{rand}$
and assumed to be independent of the stochastic primitives, for each $n$
(there is no need to let $\Ups$ vary with $n$, as the primitives
are all defined on the same probability space).
It is included in the model in order to allow the construction of randomized controls.

The process $T^n$ is said to be an {\it admissible control for the queueing control problem
	(QCP) for the $n$-th system}
if for each $j\in \calJ$, $T^n_{j}$ has sample paths in $C^{0,+}_\R[0,\iy)$
that are $1$-Lipschitz,
and the associated processes $D^n$,
$X^n$ and $I^n$, given by \eqref{40-}, \eqref{40} and \eqref{41}, satisfy \eqref{41+};
furthermore, $T^n$ is adapted to the filtration $\{\calF^n_t\}_{t\geqslant 0}$
defined by $\calF^n_t=\s\{(A^n(s),D^n(s),s\in[0,t]), \Ups\}$.
Denote by $\calA^n$ the collection of all admissible controls for the QCP
for the $n$-th system.
\begin{remark}\label{rem1}
	
	i. It is desired to support models in which the controls depend on the history
	of the queue length and cumulative idleness processes
	in addition to that of the arrival and departures. Note that our definition allows
	such dependence because, in view of \eqref{40}, \eqref{41},
	the sigma field $\s\{(A^n(s),D^n(s),X^n(s),I^n(s), s\in[0,t]), \Ups\}$
	is equal to $\calF^n_t$.
	Similarly, it is desired to allow control decisions to depend on past control decisions.
	This is allowed by our definition, where indeed $T^n_t$ can be selected
	based on any information of $T^n|_{[0,t)}$, as follows from the fact
	the process is adapted, and so such information is contained in the filtration.
	\\
	ii. The ‘auxiliary’ randomness $\Ups$ can be used to create randomized controls.
	For example, a non-preemptive policy can decide which class is assigned to
	a server when it becomes available based on the value of one RV at a time,
	taken from a sequence of IID
	decision variables with prescribed distribution. In a preemptive-resume setting,
	one can randomize, for example, the time after a new arrival
	of a job of a high priority class, at which a job of of lower priority class is preempted.
	\\
	iii. In the problem formulation of \cite{bw1,bw2}
	the control process $T^n$ may anticipate the future.
	This difference from our setting
	is due to the fact that, in the above references,
	the underlying workload control problem
	has the property that the optimal control with and without a nonanticipativity
	constraint are identical. However, this does not occur in our context.
\end{remark}

The queue length process normalized at the diffusion scale is
defined for each $i\in \calI$ by $\hat X^n_i(t)=n^{-1/2}X^n_i(t)$.
As in most earlier treatments of the PSS, we consider the model with a discounted cost.
For the $n$-th system it is given by
\begin{equation}\label{42}
	\hat J^n(T^n)=\E\int_0^\iy e^{-\gamma t}h(\hat X^n(t))dt,
	\qquad
	T^n\in\calA^n,
\end{equation}
where $h(x)=h\cdot x$ and $\gamma>0$ and $h\in(0,\iy)^I$ are constants,
and $\hat X^n$ is the rescaled queue length process associated with
the admissible control $T^n$.
The value for the $n$-th system is defined by
\[
\hat V^n=\inf\{\hat J^n(T^n):T^n\in\calA^n\}.
\]
This completes the description of the queueing model and QCP.
The complete set of problem data consists of the stochastic primitives
mentioned above and the collection of parameters
\[
\la,\,\mu,\, \hat\la,\, \hat\mu,\, C_{A},\, C_{S},\,\gamma,\,h,\, G,\, R.
\]

\subsection{The linear program, dual problem and extended heavy traffic condition}\label{sec21}

The notion of heavy traffic is tied to a critical load condition.
For systems without control  the arrival and service rates determine the traffic intensity, and hence the heavy traffic condition.
In systems with routing control, such as PSS, the traffic intensity could depend on the control policy, and the heavy traffic
condition involves an optimization problem.
In particular, for the PSS, it was proposed in \cite{harlop} to
formulate the notion of heavy traffic via a linear program (LP) that they called the `static allocation problem',
involving only first order data, which expresses
how to best allocate the servers' effort. The condition then
asserts that the LP solution yields a critical load at each server.
This LP is as follows.

Given the first order data, the LP involves the unknowns $\xi\in \R_+^J$ and $\rho\in\R$. The components $\xi_j=\xi_{ik}$ correspond to the fraction of effort server $k$ dedicates to jobs from class $i$ on the fluid scale, whereas $\rho$ is an upper bound
on the load on each server.

\noi{\it Linear Program.}	Minimize $\rho $ subject to
\begin{equation}\label{02}
	\begin{cases}
		\ds
		R\xi=\lambda,
		\\ \\
		\ds
		G\xi\leqslant \rho 1_K,
		\\ \\
		\xi\geqslant 0_J.
	\end{cases}
\end{equation}
Denote the optimal objective value (minimum achievable $\rho$) of \eqref{02}  by $\rho^*$.

\noi{\it Extended heavy traffic condition.}
$\rho^*=1$.

The extended heavy traffic condition (EHTC)
is broader than the {\it heavy traffic condition} introduced in \cite{harlop}
that has been extensively used in the literature,
in which it has also been assumed that there is only one
$\xi$ satisfying \eqref{02} with $\rho=1$.

Under the EHTC, any solution is of the form $(\xi,1)$.
Let $\slp$ denote the closed and convex subset of $\R^{J}_+$ for which
the set of all solutions is given by $\slp\times\{1\}$.
With a slight abuse of terminology, we will sometimes refer to $\slp$
as the set of solutions to the LP.
We also introduce another linear program which is the dual of \eqref{02}:

\noi{\it Dual Problem.}	Maximize $y\cdot\lambda$ subject to
\begin{equation}\label{eq:dualreformulation}
	\begin{cases}
		\ds
		\sum_k z_k=1,
		\\ \\
		\ds
		yR\leqslant zG,
		\\ \\
		\ds z\geqslant 0_K.
	\end{cases}
\end{equation}
(The standard terminology in linear programming is to refer to the LP in \eqref{02} as the {\it primal} LP. We slightly abuse this terminology by simply refering to our primal LP as `the LP'.)
A solution to \eqref{eq:dualreformulation} is given by $(y^*,z^*)\in \R^{I}\times\R_+^K$. As we will argue later, the existence of a solution to \eqref{02} with $\rho^*=1$ implies the existence of at least one solution to \eqref{eq:dualreformulation}. The following is a standing assumption in
this paper.
\begin{assumption}\label{ass:lp}
	\begin{enumerate}
		\item
		The EHTC holds. In particular, there exists at least one solution to the
		LP.
		\item For any solution $\xi\in \calS_{LP}$ and any $k\in\calK$, $(G\xi)_k=1$.
		\item There is a unique solution, $(y^*,z^*)$, to the dual problem \eqref{eq:dualreformulation}.
	\end{enumerate}
\end{assumption}

Note that the first part of the above assumption implies that for all $i\in \mathcal{I}$, $\calJ_i\neq\emptyset$. The set $\calJ_i$ being empty means no activity involves class $i$. This is not compatible with $(R\xi)_i=\lambda_i>0$. This assumption also does not restrict us on the number of solutions to the LP.
All our results are valid whether there are multiple or one solution. The meaning of the second part is that all servers are fully loaded. As argued in \cite{harlop}, this could mean that we restrict ourselves to a `bottleneck subnetwork' of critically loaded servers. 
If an original network had involved more servers, some of which are subcritical and satisfied the EHTC, the additional servers could not have involved any of the classes in the bottleneck subnetwork. Otherwise, one could construct a solution to \eqref{02} with 
$\rho<1$ from a solution with $\rho=1$. Note that the second part also implies that for any $k\in\calK$, the set $\calJ^k\neq\emptyset$.

Under the more standard assumption that the HTC holds,
activities are classified as {\it basic} and {\it non basic}
according to whether $\xi_j>0$ or not, where $\xi$ is the unique solution to the LP.
In our setting there may be multiple solutions, hence a different terminology is
required.
Call an activity $j\in\calJ$ {\it potentially basic} if $\xi_j>0$ for some $\xi\in\slp$.
An activity that is not potentially basic is called {\it always nonbasic}.
The corresponding sets of activities are denoted
$\calA_{\rm PB}$ and $\calA_{\rm AN}$.
Under Assumption \ref{ass:lp}, we are able to prove the following lemma.
\begin{lemma}\label{lem:primaldual}
	\begin{enumerate}
		\item The set $\slp=\text{\rm ch}\{\xi^{*,m}, m=1,\ldots, M\}$
		is the closed convex hull of a finite number, $M$,
		of extreme points $\xi^{*,m}\in\slp$, called {\rm modes}.
		\item One has $y^*\cdot\lambda=1$.
		\item 		For any server $k$, one has $z^*_k>0$.
		\item For any activity $j=(i,k)\in\calA_{\rm PB}$, one has
		$y^*_i\mu_{j}=z^*_k.$
		\item For any activity $j=(i,k)\in\calA_{\rm AN}$, one has
		$y^*_i\mu_{j}<z^*_k.$
		\item For any class $i\in\calI$, there exists $j\in \calJ_i\cap \calA_{\rm PB}$.
		\item 	For any class $i\in\calI$, one has $y^*_i>0$.
	\end{enumerate}
\end{lemma}
We say that the service rates $\mu$ are {\it decomposable}
if $\mu_{ik}=\al_i\beta_k$ for all $(i,k)\in\calJ$ for some constants
$\al_i$ and $\beta_k$.
In this case we will assume without loss of generality
$\sum_k \beta_k=1$. In order to motivate the relevance of systems with multiple solutions to the LP, we present the next lemma. On the one hand, decomposability of service rates lead to an explicit solution for the dual. On the other hand, they are linked to non-uniqueness of the LP.
\begin{lemma}\label{lem1}
	\begin{enumerate}
		\item If $\mu$ are decomposable,
		\[\sum_i\dfrac{\la_i}{\al_i}=\sum_k\beta_k=1.\]
		\item If $\mu$ are decomposable, the unique solution to \eqref{eq:dualreformulation} is given by  $(y^*,z^*)$ with $(y^*_i)_{i\in\calI}=(\alpha_i^{-1})_{i\in\calI}$ and $(z^*_k)_{k\in\calK}=(\beta_k)_{k\in\calK}$.
		\item If $\mu$ are decomposable and $\calJ=\calI\times\calK$
		then the LP has multiple solutions. 
		\item  If $I=K=2$ and $\calJ=\calI\times\calK$ then the LP has multiple solutions if and only if
		$\mu$ are decomposable. In addition, the number of modes $M$ is at most 2.

	\end{enumerate}
\end{lemma}
\subsection{Workload control problem, HJB equation
	and main result}\label{sec:wcpp}

Based on the notation introduced in the previous section we define the prelimit workload process by 
$$\hat W^n_t=y^*\cdot\hat X^n_t.$$
Here and throughout the paper, we fix $q\in\calI$ such that
\begin{equation}\label{145}
	h_q(y^*_q)^{-1}=\min_{i\in\calI}h_i(y^*_i)^{-1}.
\end{equation}
For any $i\in\calI$ and $j\in\calJ$, let $\sig_{A,i}:=\la_i^{1/2}C_{A_i}$, $\sig_{S,j}:=\mu_{j}^{1/2}C_{S_{j}}$, and 
\begin{equation}\label{50}
	b(\xi)=\sum_iy^*_i(\hat\la_i-\sum_{j\in \calJ_i}\hat\mu_{j}\xi_{j}),
	\qquad
	\sig(\xi)^2=
	\sum_i(y^*_i)^2(\sig_{A,i}^2+\sum_{j\in \calJ_i}\sig_{S,j}^2\xi_{j}),
	\qquad
	\xi=(\xi_{j})_{j\in\calJ}\in\slp.
\end{equation}
Let also
\begin{equation}
	b_m=b(\xi^{*,m}),\qquad \sig_m=\sig(\xi^{*,m}),\qquad m=1,\ldots M.\label{eq:param}
\end{equation}
Consider a one-dimensional controlled diffusion
with controlled drift and diffusion coefficients and reflection at the origin, given by
\begin{equation}\label{15}
	Z_t=z+\int_0^tb(\X_s)ds+\int_0^t\sig(\X_s)dB_s+L_t,
\end{equation}
where $\X$ is a control process, $B$ is a standard BM, $L$ is a reflection term
at zero, and $z\geqslant0$. A precise definition is as follows.

Given a filtration $(\calF'_t)_{t\geqslant 0}$, let $\frX((\calF'_t)_{t\geqslant 0})$ denote
the collection of $(\calF'_t)_{t\geqslant 0}$-progressively measurable
processes $\{\X\}$ taking values in $\R_+^{J}$,
for which one has $\PP(\text{for a.e.\ $t$, } \X_t\in\slp)=1$.
A tuple $\frS=(\Om',\calF',(\calF'_t)_{t\geqslant 0},\PP', B,\X,Z,L)$ is said to be
an {\it admissible control system for the WCP with initial condition $z$}
if $(\Om',\calF',(\calF'_t)_{t\geqslant 0},\PP')$ is a filtered probability space,
$B$, $\X$, $Z$ and $L$ are processes defined on it,
$B$ is a standard BM (SBM) and an $(\calF'_t)_{t\geqslant 0}$-martingale,
$\X\in\frX((\calF'_t)_{t\geqslant 0})$, $Z$ is continuous non-negative and $(\calF'_t)_{t\geqslant 0}$-adapted,
$L$ has sample paths in $C^{0,+}_\R[0,\iy)$
and is $(\calF'_t)_{t\geqslant 0}$-adapted, and equation \eqref{15} and
the identity $\int_{[0,\iy)}Z_tdL_t=0$ are satisfied $\PP'$-a.s.

Denoting by $\calA_{\rm WCP}(z)$ the collection of all control systems for the WCP
with initial condition $z$, the cost and value are defined as
\begin{equation}\label{51}
	J_\text{WCP}(z,\frS)=\E\int_0^\iy e^{-\gamma t}Z_tdt,
	\qquad
	V_\text{WCP}(z)=\inf\{J_\text{WCP}(z,\frS):\frS\in\calA_\text{WCP}(z)\}.
\end{equation}
Above, $\E$ denotes, with a slight abuse of notation, expectation w.r.t.\ the control system $\frS$.
A control system $\frS$ is said to be optimal for the WCP with initial condition
$z$ if $J_\text{WCP}(z,\frS)=V_\text{WCP}(z)$.

It is well known that the dynamic programming equation associated with
controlled diffusions as above
is of HJB type (see, for example, \cite[Section IV.5]{fle-son}).
Its role is to characterize the value function of the control problem.
In our case the HJB equation takes the following form.
For $(v_1,v_2,\xi)\in\R^2\times\slp$, let
\begin{align*}
&\bar\BH(v_1,v_2,\xi)=b(\xi)v_1+\frac{\sig(\xi)^2}{2}v_2,
\\
&\BH(v_1,v_2)=\inf_{\xi\in\slp}\bar\BH(v_1,v_2,\xi)
=\min_{m=1,\ldots, M}\bar\BH(v_1,v_2,\xi^{*,m}),
\end{align*}
where the identity on the RHS follows from the fact that both $b$ and $\sig^2$
are affine as a function of $\xi$, and $\slp=\text{\rm ch}\{\xi^{*,m}, m=1,\ldots, M\}$.
A classical solution to the HJB equation is a $C^2(\R_+:\R)$ function $u$ satisfying
\begin{equation}
	\label{14}
	\BH(u'(z),u''(z))+z-\gamma u(z)=0,
	\qquad z\in(0,\iy),
\end{equation}
and the boundary conditions at $0$ and $\iy$,
\begin{equation}\label{14+}
	u'(0)=0,
	\qquad
	|u(z)|<c(1+z), z\in\R_+, \text{ for some constant $c$}.
\end{equation}
It is well known that the
$C^2$ smoothness of the value function is tied to the question of existence
of a classical solution to the HJB equation such as \eqref{14}.
The precise statement in our setting is as follows.

\begin{proposition}\label{prop0}
	
	There exists a unique classical solution $u$ to \eqref{14}--\eqref{14+}.
	Moreover, $u=V_\text{\rm WCP}$.
\end{proposition}

A version of this diffusion control problem with $2$ modes was treated in \cite{Sheng78}.
There it was shown that the optimal control has a simple form: either one of the modes is always used, or there is a unique threshold, with
one mode used below the threshold and the other used above.
Those results are used in \cite{ourupperbound} to develop a policy and prove AO.
In Proposition \ref{prop2} we extend the condition for optimality of a single mode policy to the general setting we consider here.

The main result of this paper establishes an asymptotic lower bound.
\begin{theorem}\label{thm:lowerbound}
	$\liminf_n\hat V^n\ges V_0:=h_q(y^*_q)^{-1}V_{\rm WCP}(0)=h_q(y^*_q)^{-1}u(0)$.
\end{theorem}

\subsection{Discussion}\label{sec:degen}

The result raises the question of whether the lower bound is tight,
in the sense that there exist policies that asymptotically achieve this bound.
When this is the case, one has $\lim_n\hat V^n=V_0$.
In the discussion that follows, the terms {\it uniqueness} and
{\it multiplicity}
refer to the existence of a unique solution, and respectively,
multiple solutions to the static allocation LP.
A key issue that arises here is related to possible degeneracy of the LP solution:
A basic solution of an LP is called \emph{degenerate} if one or more basic variables are zero (cf. Definition 3.1 in \cite{gol-tod}).

Consider first the case of uniqueness.
As pointed out in \cite{harlop}, in order for CRP to be achieved using basic activities the solution must be nondegenerate.
Moreover, when uniqueness holds,
our assumption of uniqueness of dual problem solutions boils down to CRP,
and therefore Theorem~\ref{thm:lowerbound} recovers the asymptotic lower bound
from \cite{bw2}.
In this case, \cite{bw2} also constructs policies that achieve this bound, and thus
the question of tightness is answered affirmatively.
However, when the unique LP solution is degenerate,
not all servers can communicate in the graph formed by the basic activities.
In this case it is expected that the above lower bound does
not capture the correct asymptotics.

In the case of multiplicity, the situation is more interesting.
Even if all modes ( = extreme points of $\slp$ = basic optimal solutions) are degenerate, there are elements in $\slp$ for which the associated graph (where an edge between class $i$ and server $k$ exists when $\xi_j >0$ for $j=(i,k)\in\calJ$) yields communication between all servers.
As can be seen in our discussion of the HJB equation and associated optimal policies, the control used in the optimal policy indeed corresponds to extreme points of $\slp$.
If none of the extreme points is a degenerate solution to the LP, using a policy that dynamically switches between modes
as prescribed by the HJB equation should be enough to yield asymptotic optimality.
This statement has been proved
in \cite{ourupperbound} in the case $I=K=2$, $J=4$,
under multiplicity and nondegeneracy,
by constructing policies that achieve the lower bound.
As a result, the tightness question is settled in this case as well.

The difficulty one faces when
some of the modes correspond to degenerate LP solutions is that while using them it may not be possible to achieve the state space collapse needed to have the cost reach the lower bound.
Although the presence or absence of degeneracy does not affect our lower bound, a lower bound that is not asymptotically achievable is not as interesting as one that is.
However, in contrast to the case of uniqueness, in the case of
multiplicity,
the interior of $\slp$ contains nondegenerate solutions, and it should be true that
by carefully choosing controls that approach degenerate modes,
a policy achieving the lower bound can be obtained even in this case.
We therefore pose the following.

\begin{conjecture}\label{conj:multiplicity}
	Let the standing assumptions of this paper hold, and assume further that
	there are multiple solutions to the LP. Then the lower bound
	of Theorem \ref{thm:lowerbound} is tight. That is,
	there exist sequences of policies $T^n$ such that $\lim_n \hat{J}^n(T^n)=V_0$,
	and consequently $\lim_n\hat V^n=V_0$.
\end{conjecture}

We will argue in Corollary \ref{cor2} that
the work done under the uniqueness assumption carries over to a case where  multiplicity
is possible but the WCP is optimally controlled by a single mode. In particular, the tree-based threshold policy introduced by Bell and Williams \cite{bw2} based on the optimizing mode achieves the lower bound from Theorem \ref{thm:lowerbound} and thus verifies the
above conjecture in this case.

\subsection{Examples}\label{sec:ex}

We consider five examples that yield the extended heavy traffic condition with multiplicity. 
Three of the examples involve two classes and two servers, and two examples involves three classes and three servers. 

{\sc Examples (A) and (B).}
Consider a $2$-class, $2$-server system with four activities.
Label the activities $j=(i,k)$ via the correspondence $1=(1,1)$,
$2=(1,2)$, $3=(2,1)$, $4=(2,2)$. Let
\[
\mu_1=\mu_{11}=3, \qquad
\mu_2=\mu_{12}=4, \qquad
\mu_3=\mu_{21}=6, \qquad
\mu_4=\mu_{22}=8.
\]
Note that these service rates are decomposable, allowing
multiple solutions to the LP.
(Under our convention that $\beta_1+\beta_2=1$ we have $\alpha_1=7, \alpha_2=14, \beta_1=3/7,$ and $\beta_2=4/7$.)
The matrices $R$ and $G$ are given by
\[
R= \begin{pmatrix}\mu_1 & \mu_2 & 0 & 0 \\ 0 & 0 & \mu_3 & \mu_4 \end{pmatrix}
= \begin{pmatrix} 3 & 4 & 0 & 0 \\ 0 & 0 & 6 & 8 \end{pmatrix},
\qquad
G= \begin{pmatrix} 1 & 0 & 1 & 0 \\ 0 & 1 & 0 & 1 \end{pmatrix}.
\]
For the arrival rates, consider two examples, namely
\[
\text{(A)\quad $\la_1=5$, $\la_2=4$, \qquad (B) \quad $\la_1=3.5$, $\la_2=7$.}
\]
Then the linear program takes the form: Minimize $\rho$ subject to
\begin{equation}\label{400}
\begin{cases}
&3\xi_{11}+4\xi_{12}=\la_1,\\
&6\xi_{21}+8\xi_{22}=\la_2,\\
&\xi_{11}+\xi_{21}\le\rho,\\
&\xi_{12}+\xi_{22}\le\rho,\\
&\min\xi_{ik}\ge0.
\end{cases}
\end{equation}
The dual program takes the form: Maximize $y_1\la_1+y_2\la_2$ subject to
\begin{equation}\label{401}
\begin{cases}
& z_1+z_2=1,
\quad \min_kz_k\ge0,\\
&(3y_1,4y_1,6y_2,8y_2)\le(z_1,z_2,z_1,z_2).
\end{cases}
\end{equation}

Consider the set of constraints \eqref{400} of the LP, where
the third and fourth constraints hold as equalities and $\rho=1$.
By a direct calculation in the case of Example (A),
the solutions of this set are all given as convex combinations of the two
solutions
\[
\textstyle
\xi^{*,1}=(1,\frac{1}{2}, 0, \frac{1}{2})^T,
\qquad
\xi^{*,2}=(\frac{1}{3}, 1,\frac{2}{3},0)^T.
\]
In Example (B), the same is true with
\[
\textstyle
\xi^{*,1}=(1,\frac{1}{8}, 0, \frac{7}{8})^T,
\qquad
\xi^{*,2}=(0,\frac{7}{8},1,\frac{1}{8})^T.
\]

Note that the duals for Examples (A) and (B) differ only in their objective functions - the constraints are the same.
It is easy to check that $y=(\frac{1}{7},\frac{1}{14})$, $z=(\frac{3}{7},\frac{4}{7})$ satisfies the constraint \eqref{401}.
This yields an objective value of $1$ for  both Examples (A) and (B).
The duality theorem (Theorem \ref{th:duality}) shows that this solution is optimal for the dual in both examples, and, moreover,
that $\rho^*=1$. 
Hence the set of solutions to the LP is given by $\rho^*=1$
and $\slp={\rm ch}\{\xi^{*,1},\xi^{*,2}\}$, and $\xi^{*,1}$, $\xi^{*,2}$ are the modes.
Figures \ref{fig1} and \ref{fig2} present these modes in the context of
the graphs formed by classes and servers as nodes, and activities
with positive fraction of effort as edges.

These calculations show that, in both examples, the system is critically loaded
and there are multiple ways in which
the effort can be allocated so as to meet the demand. In other words,
the EHTC holds while the HTC does not.
Also, in both examples, both modes satisfy $\sum_k\xi^*_{ik}=1$, $i=1,2$,
expressing that the servers are fully occupied.
It is thus seen that Assumption \ref{ass:lp}(1)--(2) hold.

\begin{figure}[h]
	\centering
	\vspace{.8em}
	\begin{tikzpicture}[scale=1]
		\node[shape=circle,draw=black] (A) at (-1,1) {1};
		\node[shape=circle,draw=black] (B) at (1,1) {2};
		\node[shape=circle,draw=black] (C) at (1,-1) {2};
		\node[shape=circle,draw=black] (D) at (-1,-1) {1};
		\path [thick](B) edge node[right] {} (C);
		\path [thick](D) edge node[left] {\quad\quad} (A);
		\path[thick](A)edge node[left]{}(C);
		
		\node[shape=circle,draw=black] (E) at(4,1) {1};
		\node[shape=circle,draw=black] (F) at (6,1) {2};
		\node[shape=circle,draw=black] (G) at (6,-1) {2};
		\node[shape=circle,draw=black] (H) at (4,-1) {1};
		\path [thick] (F) edge node[left] {} (H);
		\path [thick](H) edge node[left] {\quad\quad} (E);
		\path[thick](E)edge node[left]{}(G);
		
	\end{tikzpicture}
	\\
	\flushleft
	\vspace{-10.7em}
	
	arrival rates $\la_i$
	\hspace{5.1em} $5$ \hspace{5.3em} $4$
	\hspace{8.4em} $5$ \hspace{5.3em} $4$
	
	\vspace{.5em}
	
	class labels $i$
	
	\vspace{3.1em}
	
	service rates $\mu_j$
	\hspace{4em} $3$ \hspace{3em} $4$ \hspace{2.6em} $8$
	\hspace{7.2em} $3$ \hspace{.9em} $6$ \hspace{2.8em} $4$
	
	\vspace{.7em}
	
	server labels $k$
	
	\vspace{.9em}
	
	fractions of effort $\xi_{ik}$
	\hspace{2em} $1$ ($0$)
	\hspace{3.7em} $\frac{1}{2}$\ \ $\frac{1}{2}$
	\hspace{6.8em} $\frac{1}{3}$ \ \ $\frac{2}{3}$ \hspace{3.5em} $1$ ($0$)

\vspace{.7em}

	\hspace{13.5em} mode 1 \hspace{12.1em} mode 2
	\caption{\label{fig1}\sl
	The two modes in Example (A).
	}
\end{figure}
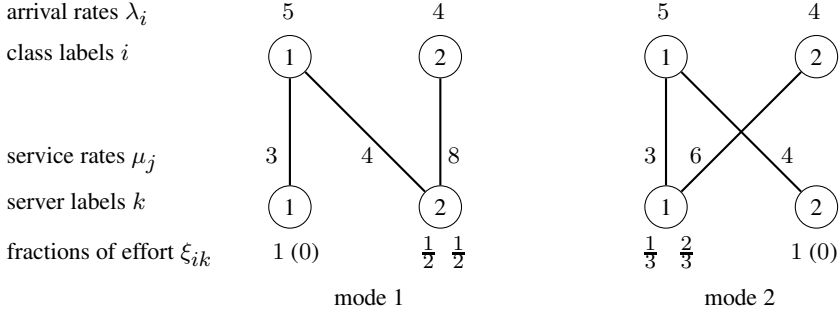

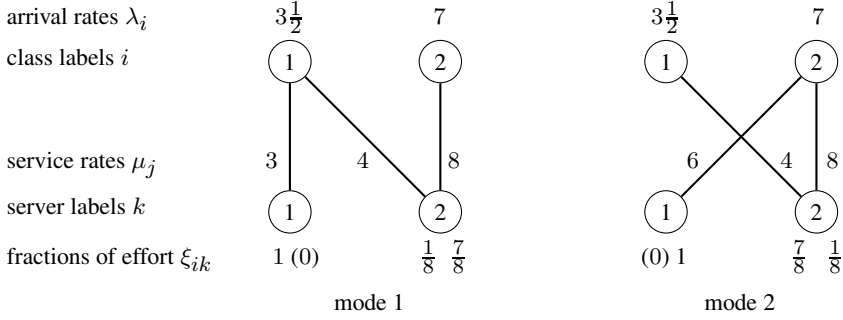
\begin{figure}[h]
	\centering
	\vspace{.8em}
	\begin{tikzpicture}[scale=1]
		\node[shape=circle,draw=black] (A) at (-1,1) {1};
		\node[shape=circle,draw=black] (B) at (1,1) {2};
		\node[shape=circle,draw=black] (C) at (1,-1) {2};
		\node[shape=circle,draw=black] (D) at (-1,-1) {1};
		\path [thick](B) edge node[right] {} (C);
		\path [thick](D) edge node[left] {\quad\quad} (A);
		\path[thick](A)edge node[left]{}(C);
		
		\node[shape=circle,draw=black] (I) at (4,1) {1};
		\node[shape=circle,draw=black] (J) at (6,1) {2};
		\node[shape=circle,draw=black] (K) at (6,-1) {2};
		\node[shape=circle,draw=black] (L) at (4,-1) {1};
		\path [thick](J) edge node[right] {} (K);
		\path  [thick](J) edge node[left] {} (L);
		\path[thick](I)edge node[left]{}(K);
		
	\end{tikzpicture}
	\\
	\flushleft
	\vspace{-10.7em}
	
	arrival rates $\la_i$
	\hspace{4.8em} $3\frac{1}{2}$ \hspace{4.9em} $7$
	\hspace{8.1em} $3\frac{1}{2}$ \hspace{5.em} $7$
	
	\vspace{.5em}
	
	class labels $i$
	
	\vspace{3.1em}
	
	service rates $\mu_j$
	\hspace{4em} $3$ \hspace{2.8em} $4$ \hspace{2.8em} $8$
	\hspace{9.em} $6$ \hspace{2.9em} $4$ \hspace{.9em} $8$
	
	\vspace{.7em}
	
	server labels $k$
	
	\vspace{.9em}
	
	fractions of effort $\xi_{ik}$
	\hspace{2em} $1$ ($0$)
	\hspace{3.7em} $\frac{1}{8}$\ \ $\frac{7}{8}$
	\hspace{6.8em} ($0$) 1 \hspace{3.8em} $\frac{7}{8}$ \ \ $\frac{1}{8}$

\vspace{.7em}

	\hspace{13.5em} mode 1 \hspace{12.1em} mode 2
	\caption{\label{fig2}\sl
	The two modes in Example (B).
	}
\end{figure}

To complete the verification of our standing assumption
it remains to show that Assumption \ref{ass:lp}(3) holds, i.e.,
that the dual program has a unique solution.
Our argument works for both Examples (A) and (B).
Note that in both examples, a strictly positive quantity of each of the four activities is used in at least one of the modes.
Thus, by complementary slackness (Lemma \ref{lem:comp-slack}(iii)), the constraints in the 
second line of \eqref{401} all hold at equality.
(This argument is carried out in  more detail in the proof of part 4 of Lemma \ref{lem:primaldual} in \S \ref{sec3}.)
This yields $18 y_1 y_2=z_1^2$ and $32y_1y_2=z_2^2$, so that $z_2=\frac{4}{3} z_1$.
Combined with $z_1+z_2=1$ this yields a unique solution to the dual.

Going back to the graphs shown in Figures \ref{fig1} and \ref{fig2},
it is seen that the two examples differ in the way the two modes are paired.
In the case of Example (A), the non-basic activity (i.e., the one for which $\xi_j=0$)
switches a server between the two modes, whereas it switches a class
in Example (B). Although this distinction has no consequences in this paper,
it turns out to be of crucial significance when it comes to implementing
AO policies, as in \cite{ourupperbound}. As is also shown in \cite{ourupperbound},
it is impossible to have a pairing of modes in which the non-basic activity
switches both a class and a server.
Furthermore, the results of \cite{ourupperbound},
which show that the lower bound of Theorem \ref{thm:lowerbound}
is tight, cover these two examples.

\skp

{\sc Example (C).}
Keep $\mu_j$, $j=1,\ldots,4$ as above, but let $\la_1=4$ and $\la_2=6$.
In this case the EHTC still holds, with 
\[
\textstyle
\xi^{*,1}=(0,1, 1, 0)^T,
\qquad
\xi^{*,2}=(1,\frac{1}{4},0,\frac{3}{4})^T.
\]

This is an example of a degenerate LP as discussed above.
Note that the  dual solution is the same as in Examples (A) and (B), and is unique.
In particular, this example satisfies our standing assumption
and so the lower bound of Theorem \ref{thm:lowerbound} is valid.
Because of the LP degeneracy, the results of \cite{ourupperbound}
do not cover this example, and because, further, there are
multiple solutions to the LP, this example falls into the category
of problems to which Conjecture \ref{conj:multiplicity} refers,
stating that nevertheless this lower bound is tight.

\skp

{\sc Example (D).}
We next provide an example that goes beyond uniqueness of solutions
to the dual.
We consider a 3-class, 3-server system with seven activities. Let 
\[
\mu_{11}=3, ~~\mu_{12}=4,  ~~\mu_{21}=6,  ~~\mu_{22}=8, ~~\mu_{23}=6, ~~\mu_{32}=7, ~~\mu_{33}=6.
\]
Although the services rates are  not decomposable (because $\frac{\mu_{22}}{\mu_{23}}\neq\frac{\mu_{32}}{\mu_{33}}$), the service rates in the 2 by 2 
subsystem consisting of classes 1 and 2, and  servers 1 and 2, are decomposable.

Then the matrices $R$ and $G$ are given by
\[
R= \begin{pmatrix} 3 & 4 & 0 & 0 & 0 & 0 & 0\\ 0 & 0 & 6 & 8 & 6 & 0 & 0\\
0 & 0 & 0 & 0 & 0 & 7 & 6\end{pmatrix},
\qquad
G= \begin{pmatrix} 1 & 0 & 1 & 0 & 0 & 0 & 0\\ 0 & 1 & 0 & 1  & 0 & 1 & 0\\
 0 & 0 & 0 & 0 & 1 & 0 & 1\end{pmatrix}.
\]

We consider the arrival rates $\lambda_1=5, \lambda_2=4, \lambda_3=6$.
The LP in this case is: 
Minimize $\rho$ subject to
\begin{equation}\label{500D}
\begin{cases}
&3\xi_{11}+4\xi_{12}=5,\\
&6\xi_{21}+8\xi_{22}+ 6\xi_{23}=4,\\
&7\xi_{32}+ 6\xi_{33}=6,\\
&\xi_{11}+\xi_{21}\le\rho,\\
&\xi_{12}+\xi_{22}+\xi_{32}\le\rho,\\
&\xi_{23}+\xi_{33}\le\rho,\\
&\min_{ik}\xi_{ik}\ge0.
\end{cases}
\end{equation}
The dual program takes the form: Maximize $5y_1+4y_2+6y_3$ subject to
\begin{equation}\label{401D}
\begin{cases}
& z_1+z_2+z_3=1,
\quad \min_kz_k\ge0,\\
&(3y_1,4y_1,6y_2,8y_2, 6y_2,7y_3,6y_3)\le(z_1,z_2,z_1,z_2,z_3,z_2,z_3).
\end{cases}
\end{equation}

Let
$$\xi=(\xi_{11},\xi_{12},\xi_{21},\xi_{22},\xi_{23},\xi_{32},\xi_{33}).$$
The optimal objective value is $\rho^*=1$, and the set $\slp$  is again a line segment with extreme points
\[
\textstyle
\xi^{*,1}=(1,\frac{1}{2}, 0, \frac{1}{2},0,0,1)^T,
\qquad
\xi^{*,2}=(\frac{1}{3}, 1,\frac{2}{3},0,0,0,1)^T.
\]
Note that $\xi^*_{23}=\xi^*_{32}=0$ in both modes, so both modes correspond to the network 
separating into two subnetworks.

In this case the dual also has multiple solutions. In particular, if we take $\eta \in [\frac{3}{10}, \frac{24}{73}]$, and
$$z_1=\frac{3}{7}(1-\eta), \qquad z_2=\frac{4}{7}(1-\eta), \qquad  z_3=\eta,$$
$$y_1=\frac{1}{7}(1-\eta), \qquad y_2=\frac{1}{14}(1-\eta), \qquad  y_3=\frac{1}{6}\eta,$$
then $(y,z)$ solves the dual.

The duality theorem provides an easy verification that the above values of $\xi, y,z$ are optimal for their respective problems. 
We simply need to verify that they are feasible and yield the same objective value of $1$. 
This verification follows by simply plugging in the values given above. 

We can see that the above $2$ modes are the only ones as follows.
The fact that $\xi^*_{23}=\xi^*_{32}=0$ in all modes follows from complementary slackness:
for $\eta =0.32 \in (\frac{3}{10}, \frac{24}{73}]$ we have $6y_2 < z_3$ and $7y_3 < z_2$,
so complementary slackness implies that $\xi^*_{23}=\xi^*_{32}=0$.
It is then immediate from \eqref{500D} that $\xi^*_{33}=1$.  With $\xi^*_{33}=1, \xi^*_{23}=\xi^*_{32}=0$, 
\eqref{500D} reduces to \eqref{400} with $\lambda_1=5, \lambda_2=4$. 
Note that the first $4$ entries of $\xi^{*,1}$ and $\xi^{*,2}$ match the corresponding values in Example (A).

This system satisfies parts $1$ and $2$ of Assumption \ref{ass:lp}, but not part $3$.
Multiplicity of the dual solution, roughly speaking, implies that the workload process is not one-dimensional.
PSS of this type are considered in \cite{pes16}, but this particular example is excluded from their treatment because they assume that the  solution 
of the primal LP is unique, and that is not true in this example.

\skp
  
  
{\sc Example (E).}
Finally, we consider the 3-class, 3-server system of Example (D) with slightly modified service rates. Let 
\[
\mu_{11}=3, ~~\mu_{12}=4,  ~~\mu_{21}=6,  ~~\mu_{22}=8, ~~\mu_{23}=6, ~~\mu_{32}=8, ~~\mu_{33}=6.
\]
Here, the services rates are decomposable, with $\alpha_1=10, \alpha_2=20, \alpha_3=20, \beta_1=3/10, \beta_2=4/10,$ and $\beta_3=3/10$. 
The matrices $R$ and $G$ are given by
\[
R= \begin{pmatrix} 3 & 4 & 0 & 0 & 0 & 0 & 0\\ 0 & 0 & 6 & 8 & 6 & 0 & 0\\
0 & 0 & 0 & 0 & 0 & 8 & 6\end{pmatrix},
\qquad
G= \begin{pmatrix} 1 & 0 & 1 & 0 & 0 & 0 & 0\\ 0 & 1 & 0 & 1  & 0 & 1 & 0\\
 0 & 0 & 0 & 0 & 1 & 0 & 1\end{pmatrix}.
\]

We consider the same arrival rates as in Example (D):  $\lambda_1=5, \lambda_2=4, \lambda_3=6$.
The LP in this case is: 
Minimize $\rho$ subject to
\begin{equation}\label{500E}
\begin{cases}
&3\xi_{11}+4\xi_{12}=5,\\
&6\xi_{21}+8\xi_{22}+ 6\xi_{23}=4,\\
&8\xi_{32}+ 6\xi_{33}=6,\\
&\xi_{11}+\xi_{21}\le\rho,\\
&\xi_{12}+\xi_{22}+\xi_{32}\le\rho,\\
&\xi_{23}+\xi_{33}\le\rho,\\
&\min_{ik}\xi_{ik}\ge0.
\end{cases}
\end{equation}
The dual program takes the form: Maximize $5y_1+4y_2+6y_3$ subject to
\begin{equation}\label{401E}
\begin{cases}
& z_1+z_2+z_3=1,
\quad \min_kz_k\ge0,\\
&(3y_1,4y_1,6y_2,8y_2, 6y_2,8y_3,6y_3)\le(z_1,z_2,z_1,z_2,z_3,z_2,z_3).
\end{cases}
\end{equation}

As before, let
$$\xi=(\xi_{11},\xi_{12},\xi_{21},\xi_{22},\xi_{23},\xi_{32},\xi_{33}).$$
The optimal objective value is $\rho^*=1$, and there are multiple solutions to the LP, with the set $\slp$ having $3$ extreme points:
\[
\textstyle
\xi^{*,1}=(1,\frac{1}{2}, 0, \frac{1}{2},0,0,1)^T,
\qquad
\xi^{*,2}=(\frac{1}{3}, 1,\frac{2}{3},0,0,0,1)^T,
\qquad
\xi^{*,3}=(1, \frac{1}{2},0,0,\frac{2}{3},\frac{1}{2},\frac{1}{3})^T.
\]
It follows from the fact that 
the columns of $\begin{pmatrix} R\\G \end{pmatrix}$  corresponding to positive elements of $\xi$ are linearly independent
that these are all extreme points.
It can be seen that these are the only extreme points as follows.
Turn the constraints in \eqref{500E} into a set of $6$ equations by replacing the $3$ inequalities with equalities and replacing $\rho$ by $1$.
The set of non-negative solutions to this set of equations is precisely $\slp$.
It is easy to see that $\slp$ can be parametrized by $\xi_{11}$ and $\xi_{33}$: The five other $\xi_{ik}$ can be determined once 
$\xi_{11}$ and $\xi_{33}$ are known. Furthermore, we must have $\xi_{11} \in [1/3,1], \xi_{33} \in [1/3,1]$, and $\xi_{11}+\xi_{33} \ge 4/3$.
These $3$ constraints determine a triangle in the positive quadrant, where the extreme points correspond precisely to 
$\xi^{*,1}, \xi^{*,2}$, and $\xi^{*,3}$.

In this case the dual  has a unique solution (which corresponds to  $\eta =\frac{3}{10}$ in Example (D)):
$$z_1=\frac{3}{10}, \qquad z_2=\frac{4}{10}, \qquad  z_3=\frac{3}{10},$$
$$y_1=\frac{1}{10}, \qquad y_2=\frac{1}{20}, \qquad  y_3=\frac{1}{20}.$$
(Uniqueness can be established in a manner similar to Examples (A) and (B).)
The duality theorem again provides an easy verification that the above values of $\xi, y,z$ are optimal for their respective problems.

This system satisfies all three parts of Assumption \ref{ass:lp}, so that Proposition \ref{prop0} applies to it,
but it is not covered by  \cite{ourupperbound}, which is restricted to 2 class, 2 server systems.
It is also worth noting that modes 1 and 2 are both degenerate.


\skp

Whereas the examples provided above are concerned only
with first order data, the second order data determines the structure of the
WCP as well as the HJB equation. In \S \ref{sec43}, Examples~(A),
(C) and (E) above are
further developed, discussing the role of the second order data.

\section{The LP under the EHTC}\label{sec3}
\beginsec
The goal of this section is to prove Lemmas \ref{lem:primaldual}
and \ref{lem1}.
In the proof of Lemmas \ref{lem:primaldual} we use several standard results related to linear programs, such as complementary slackness and the duality theorem. 
We also use a slightly less standard result, commonly referred to `strict complementary slackness'.
For all of these we refer to appropriate results in \cite{schrijver}. For ease of reference all of our cited results are reproduced in the appendix.

\begin{proof}[Proof of Lemma \ref{lem:primaldual}.]

1. The set $\slp$ is the set of $\xi$ such that the constraints in \eqref{02} hold with $\rho=1$. 
By (9) in Section 7.2 of \cite{schrijver} (Definition \ref{def:con-poly}), this set is a convex polyhedron. We show that this polyhedron is bounded. 
Consider activity $j=(i,k)\in\calJ$. Since for any $\xi\in \slp$ we have $(G\xi)_k= 1$, for any $k\in \calK$, $0\leqslant \xi_j\leqslant 1$.  By Corollary 7.1c of \cite{schrijver} (Lemma \ref{lem:polytope}), this is enough to show that it is a polytope and thus finitely generated.

2.
We first rearrange \eqref{02} to a more standard form that enables the use of results from \cite{schrijver}. To that effect, let \begin{align*}
	\bar{x}&=(\rho,\xi_1,\ldots \xi_J)^T,\\
	\bar{c}&=(-1,0_J),\\
	\bar{b}'&=(\lambda,-\lambda,0_K,0_J),		
\end{align*}
and 
\begin{equation*}
	\bar{G}=\begin{pmatrix}
		
		0_I & R
		\\
		
		0_I & -R
		\\
		-1_K & G
		\\
		0_J & -I_J
	\end{pmatrix}.
\end{equation*}
Then \eqref{02} becomes:
\begin{equation}
	\text{maximize } \bar{c}\cdot\bar{x} \text{ subject to } \bar{G}\bar{x}\leqslant \bar{b}. \label{eq:primal}
\end{equation}
Then, directly,
\begin{equation*}
	\bar{c}\cdot\bar{x}=-\rho.
\end{equation*}
In addition, $\bar{G}\bar{x}\leqslant \bar{b}$ is equivalent to following the set of conditions:
\begin{align*}
	R\xi&\leqslant \lambda,\\
	-R\xi&\leqslant -\lambda,\\
	-\rho\cdot\mathds{1}+G\xi&\leqslant 0_K\\
	-	\xi&\leqslant 0_J.
\end{align*}
As a result, the LP \eqref{eq:primal}, with the above parameters $(\bar b,\bar c,\bar G)$,
is equivalent to our original LP \eqref{02}.
We now use Corollary 7.1g of \cite{schrijver} (Theorem \ref{th:duality}), according to which
the dual of \eqref{eq:primal} is given by
the LP
\begin{equation}	\text{Minimize } y\cdot\bar{b} \text{ subject to } y\bar{G}=\bar{c},\, y\geqslant 0,\label{eq:dual2}
\end{equation}
and the primal and the dual have the same value.
(It follows from Assumptions \ref{ass:lp} (1) and (3) that both sets in \eqref{eq:lpduality} are non-empty.)

Therefore result 2 will be proved once we show that the LP \eqref{eq:dual2}
is equivalent to our original dual, \eqref{eq:dualreformulation}.
To this end, write $y$ as $y=(y^1,y^2,z,s)$, where $y^1,y^2\in\R_+^I$,
$z\in\R_+^K$ and $s\in\R_+^J$.
Next, note that by plugging in the values for $y$ and $\bar{b}$, we obtain 
$$y\cdot\bar{b}=(y^1-y^2)\cdot\lambda.$$
This  equality means that minimizing $y\cdot\bar{b}$ is the same as maximizing $(y^2-y^1)\cdot\lambda$.
In addition, $y\bar{G}= \bar{c}$ is equivalent to the following set of conditions:
\begin{align*}
	\sum_k z_k&=1,\\
	(y^2-y^1)R&=zG-s.
\end{align*}
Set $\bar{y}=y^2-y^1$. Note that the sign of $\bar{y}$ is unknown. Since $s\geqslant 0$, we can replace the last equality by an inequality.
This is sufficient to conclude that \eqref{eq:dual2} and \eqref{eq:dualreformulation}
are equivalent.
This shows $y^*\cdot\lambda=\rho^*=1$.

3.
We appeal to result (36) in chapter 7 of \cite{schrijver} (Theorem  \ref{th:app}), commonly referred to as `strict complementary slackness'.  By Assumption \ref{ass:lp},
we have $(G\xi^*)_k= 1$ for all $k\in\calK$.
The dual variable corresponding to the constraint $(G\xi^*)_k\leqslant \rho$ is $z^*_k$.
Since(36) (i) of \cite{schrijver} (Theorem \ref{th:app} (a)) does not hold for any solution of the primal, (36) (ii) (Theorem \ref{th:app} (b)) must hold.
Thus $z_k^*>0$ for all $k$.

4.
In a manner similar to the rearrangement/translation in part 2 above, we can translate our primal and dual respectively 
into the minimization and maximization problem in \eqref{eq:lpduality}. 
(In this translation, denoting  the variable $y$ appearing in \eqref{eq:lpduality}  by $\tilde{y}$ to avoid confusion with our variable $y$ in \eqref{eq:dualreformulation},
we have $\tilde{y}=(\xi,\rho,\tilde{s})$, where $\tilde{s} \in \R^{K}_+$ corresponds to slack variables for the $K$ constraints
$G\xi\leqslant \rho 1_K$.)
Invoking (35)(iii)  of \cite{schrijver} (Lemma \ref{lem:comp-slack}), which is complementary slackness,
if there is a solution of \eqref{eq:primal} such that $\xi^*_j>0$, then 
$(y^*R)_j=(z^*G)_j$. In other words, 
\begin{equation}\label{c3}
	y^*_i\mu_{j}=z^*_k \quad \text{whenever $j=(i,k)\in\calA_{\rm PB}$.}
\end{equation}

5.
We use the maximization/minimization pair of part 4, and invoke (36) of  \cite{schrijver} (Theorem \ref{th:app}) again, this time to say that 
if $\xi^*_j=0$ for all solutions to \eqref{eq:primal} then 
the associated constraint 
$(y^*R)_j \leqslant (z^*G)_j$ has positive slack, so that
\begin{equation}\label{c4}
       y^*_i\mu_{j}<z^*_k \quad \text{whenever $j=(i,k)\in\calA_{\rm AN}$.}
\end{equation} 

6.
Assume  that $i$ is such that $\calJ_i\cap \calA_{\rm PB}=\emptyset$. Then for any $\xi \in \slp$, we obtain $(R\xi)_i=\sum_{j\in \calJ_i}\mu_j\xi_j=0$. Since $\lambda>0$ there is a contradiction with $(R\xi)_i = \lambda_i$ in \eqref{02}, so such an $i$ cannot exist.

7.
By  part 6, for every $i$, there exists a server $k$ such that $j=(i,k)\in \calA_{\rm PB}$. By part 4,
$$y^*_i=\dfrac{z^*_k}{\mu_{j}},$$
and by part 3 the last quantity is positive.
\end{proof}

\begin{proof}[Proof of Lemma \ref{lem1}.]

1. For any $\xi\in\slp$ and $i\in\calI$, 
\begin{align*}
	\dfrac{\lambda_i}{\al_i}&=\sum_{j\in \calJ_i} \dfrac{\mu_{j}}{\alpha_i}\xi_{j}\\
	&=\sum_{k, (i,k)\in \calJ}\beta_k\xi_{ik}
\end{align*}
Summing over $i$, we get
\begin{align*}
	\sum_i \dfrac{\lambda_i}{\al_i}&=\sum_i \sum_{k, (i,k)\in \calJ}\beta_k\xi_{ik}\\
	&=\sum_k\beta_k\sum_{j\in \calJ^k} \xi_j\\
	&=\sum_k\beta_k.
\end{align*}
The last equality is obtained using Assumption \ref{ass:lp}(2).
The claim follows from the convention $\sum_k\beta_k=1$ that we have imposed.

2. The choice $z^*=\beta$ obviously satisfies the condition of summing to 1. For any $j=(i,k)\in \calJ$, $(y^*R)_j=\dfrac{\mu_j}{\al_i}=\beta_k= (z^*G)_j$ . In addition, by part 1,  the optimal value for the primal is attained by this solution. Since both problems have the same value we have shown that 
the choice given for 
$(y^*,z^*)$ is a solution to the dual. The solution is unique by Assumption \ref{ass:lp}.

3. Let $\xi^*$ be such that for any $j\in \calJ_i$, $\xi^*_{j}=\frac{\lambda_i}{\al_i}$. One can check that when $\calJ=\calI\times\calK$, $\xi^*\in\slp$. Indeed, for any $i$, 
\[(R\xi^*)_i=\sum_{j\in \calJ_i}\mu_j\xi^*_j=\sum_{k, (i,k)\in \calJ}\dfrac{\lambda_i}{\alpha_i}\alpha_i\beta_k=\lambda_i\sum_{k}\beta_k=\la_i,\]
where the last equality comes from the definition of $\beta$. Similarly,
\[(G\xi^*)_k=\sum_{j\in \calJ^k}\xi^*_j=\sum_{i} \dfrac{\lambda_i}{\alpha_i}=1,\]
where the last equality comes from part 1.

Notice that all the components of $\xi^*$ are positive. We will now perturb this solution to get another element of $\slp$. Let $\delta\in \R^J$ such that $\delta_{ik}=\frac{(-1)^{i+k}}{\beta_k}$ if $i$ and $k$ are no larger than $2$
and $\delta_{(i,k)}=0$ otherwise. For $\eps>0$ small enough, $\xi^*+\eps\delta\in \R_+^J$. 
Server 1 diverts a fraction $\frac{\eps}{\beta_1}$ of its effort from class 2 to class 1, server 2 diverts a fraction $\frac{\eps}{\beta_2}$ of its effort from class 1 to class 2. 
We thus have $R\delta=G\delta=0$, so that this solution satisfies the constraints of the LP.  To see this, note that 
due to decomposability, we have
\begin{align*}
	(R\delta)_1&=\dfrac{\mu_{11}}{\beta_1}-\dfrac{\mu_{12}}{\beta_2}=\alpha_1-\alpha_1=0,\\
	(G\delta)_1&=\dfrac{1}{\beta_1}-\dfrac{1}{\beta_1}=0.
\end{align*}
The second coordinate is obtained in the same way. The other coordinates are also zero because $\delta_{ik}=0$ if $i$ or $k$ is neither 1 or 2.
This proves $\xi^*+\eps\delta\in\slp$ which proves decomposability is sufficient for multiplicity. 

4. We have already shown that decomposability is sufficient for multiplicity in general, we now prove the necessary part, that is when $I=K=2$, having multiple solution can only happen if $\calJ=\calI\times\calK$. We identify $\calI$ and $\calJ$ with $\lbrace 1,2\rbrace$. Let us have two distinct solutions $\xi^{(p)}$,
$p=1,2$. By Assumption \eqref{ass:lp}, $\sum_{j\in \calJ^k}\xi^{(p)}_{j}=1$ holds for $k=1,2$ and $p=1,2$.
Next, define $\del_{ik}=\xi^{(1)}_{ik}-\xi^{(2)}_{ik} $. Then, since both $\xi^{(1)}$ and $\xi^{(2)}$ are solutions to \eqref{02}, 
\begin{align}\label{002}
	&
	\del_{11}\mu_{11}+\del_{12}\mu_{12}=0,
	\\& \notag
	\del_{21}\mu_{21}+\del_{22}\mu_{22}=0,
	\\& \notag
	\del_{11}+\del_{21}=0,
	\\& \notag
	\del_{12}+\del_{22}=0.
\end{align}
The last three equations imply
\begin{equation}\label{003}
	\del_{11}\mu_{21}+\del_{12}\mu_{22}=0.
\end{equation}
Since the solutions are distinct, one of the four entries of $\del$ is nonzero,
and it follows that all four entries of $\del$ are nonzero (this uses nonvanishing of $\mu_{ik}$).
By \eqref{002} and \eqref{003}
we have
\begin{equation}\label{02:}
	\frac{\mu_{11}}{\mu_{12}}=\frac{\mu_{21}}{\mu_{22}}.
\end{equation}
Hence $(\mu_{21},\mu_{22})=c(\mu_{11},\mu_{12})$
for some constant $c>0$. This gives the decomposability.

Showing that the number of modes is at most $2$ amounts to showing
that when the LP has multiple solutions, $\slp$ is a line segment.
Using the four equations in \eqref{002} along with \eqref{02:},
it is not hard to show that whenever $\xi$ and $\eta$ are in $\slp$,
one has $\xi-\eta\in{\rm span}(\del^0)$, where $\del^0$ is the vector
$(\delta_{(i,k)}^0)_{(i,k)\in\calJ}=(\dfrac{\mu_{i1}}{\mu_{ik}}(-1)^{i+k})_{(i,k)\in\calJ}$.
Since $\slp$ is convex and bounded,
this shows that it forms a line segment,
and proves the claim.
\end{proof}

\section{The BCP, WCP and HJB equation}\label{sec4}
\beginsec

The main goal of this section is to derive the BCP and WCP
describing the asymptotics of the model, and construct
an optimal solution to the latter via the HJB equation.
By way of doing this, we introduce rescaled processes
and develop relations among them,
that are used in the remainder of the paper.
The derivation of the BCP and WCP appears in \S\ref{sec41}.
\S\ref{sec42} states and proves Lemma \ref{lem3}, which constructs a candidate
optimal solution to the WCP.
\S \ref{sec43} proves Proposition \ref{prop0} and affirms,
in Corollary \ref{cor1}, that the aforementioned candidate is optimal. It also addresses,
in Proposition \ref{prop2}, the special case where there is one mode for which
both coefficients $b$ and $\sigma$ are smallest,
showing that it is optimal to always use this mode.

\subsection{Derivation of the BCP and WCP}\label{sec41}

Two control problems formulated in terms of diffusion processes,
referred to as the BCP and the WCP, are presented here, aimed
at formally describing the asymptotics of the queue length
and of the workload, respectively. The WCP has the advantage
of lying in dimension 1, and it is this problem that plays a central role
in this paper.
In earlier work on the PSS model, the BCP was obtained from
the prelimit problem and then transformed into a WCP.
In this paper, the connection between the prelimit and the WCP
is made directly, and therefore the BCP itself has no use in the proofs.
It is presented for the sole purpose of comparison to the form
it took in earlier work.
However, there will be reference to some of the notation
introduced and equations obtained along the way to deriving it.

\subsubsection{BCP derivation}\label{sec411}
Following standard convention, rescaled processes at the diffusion scale are
denoted with hat.
In addition to the ones already defined, $\hat X^n$ and $\hat W^n$,
let
\begin{align}\label{004}
	\hat A^n_i(t) &= n^{-1/2}(A^n_i(t)-\la^n_it),
	\\ \notag
	\hat S^n_{j}(t) &= n^{-1/2}(S^n_{j}-\mu^n_{j}t),
	\\\hat I^n_k(t)&=n^{1/2}I^n_k(t),
	\\\hat L^n(t)&=z^*\cdot \hat I^n(t).
	\notag
\end{align}
By \eqref{40-} and the balance equation \eqref{40},
\begin{align}\label{eq:balancex}
	\hat X^n_i(t) &=\hat A^n_i(t)+n^{-1/2}\la^n_it
	-\sum_{j\in \calJ_i}\hat S^n_{j}(T^n_{j})
	-\sum_{j\in \calJ_i}n^{-1/2}\mu^n_{j}T^n_{j}(t).
\end{align}
We next introduce a notion of which no rigorous use will be made
in this paper,
serving only for the derivation of
the BCP: A {\it fluid control process} is
a process $\bar T^n$ that takes the form
$\bar T^n_{j}=\int_0^\cdot\bar\X^n_{j}(s)ds$, where $\bar\X^n$ is some  process
taking values in $\slp$.
Given any fluid control process $\bar T^n$ we can perform the centering
\[
\hat Y^n_{j}(t)=-n^{1/2}(T^n_{j}(t)-\bar T^n_{j}(t)).
\]
Then, by \eqref{eq:balancex} and the definition of $\hat\la^n$, $\hat\mu^n$ one has
\begin{align*}
	\hat X^n_i(t)&=\hat A^n_i(t)-\sum_{j\in \calJ_i}\hat S^n_{j}(T^n_{j}(t))
	+\sum_{j\in \calJ_i}\mu_{j}\hat Y^n_{j}(t)
	\\&
	\qquad
	-\sum_{j\in \calJ_i}\hat\mu^n_{j} T^n_{j}(t)-\sum_{j\in \calJ_i}n^{1/2}\mu_{j}\bar T^n_{j}(t)+n^{-1/2}\la^n_it.
\end{align*}
The sum of the last two terms above is equal to
\begin{align*}
	\int_0^t n^{-1/2}\Big(\la^n_i-n^{1/2}\sum_{j\in \calJ_i}\mu_{j}\bar\X^n_{j}(s)\Big)ds
	&=n^{1/2}\int_0^t\Big(\la_i-\sum_{j\in \calJ_i}\mu_{j}\bar\X^n_{j}(s)\Big)ds
	+\hat\la^n_it\\
	&=\hat\la^n_it,
\end{align*}
where the last equality follows from the fact that
$\bar\X^n(t)$ takes values in $\slp$.
Thus
\begin{equation}\label{04}
	\hat X^n_i(t)=\hat A^n_i(t)-\sum_{j\in \calJ_i}\hat S^n_{j}(T^n_{j}(t))+\hat{\lambda}^n_i t-\sum_{j\in \calJ_i}\hat{\mu}^n_{j}T^n_{j}(t)+\sum_{j\in \calJ_i}\mu_{j} \hat Y^n_{j}(t).
\end{equation}
Next,
by \eqref{41} and the identity $\sum_{j\in \calJ^k}\bar\X^n_{j}(t)=1$,
\begin{equation}\label{05}
	\hat I^n_k(t)=n^{1/2}\sum_{j\in \calJ^k}(\bar T^n_{j}(t)-T^n_{j}(t))
	=\sum_{j\in \calJ^k}\hat Y^n_{j}(t).
\end{equation}
By the central limit theorem for renewal processes
\cite[\S 17]{bill}, the processes $(\hat A^n_i,\hat S^n_{j})$
converge to $(A_i,S_{j})$,
a collection of $I+J$ mutually independent BMs starting from zero,
with zero drift and diffusion coefficients given by the constants $\la_i^{1/2}C_{A_i}$
and $\mu_{j}^{1/2}C_{S_{j}}$, which we have earlier denoted by $\sig_{A,i}$
and $\sig_{S,j}$, respectively.
Assuming, without any justification as this point, that
$(\hat A^n,\hat S^n, \hat X^n, T^n, \bar T^n, \hat Y^n,\hat I^n)$
weakly converge to $(A,S,X,T,T,Y,I)$,
we have from \eqref{04} and \eqref{05}
\begin{align}\label{06}
	X_i(t)
	&=A_i(t)-\sum_{j\in \calJ_i}S_{j}(T_{j}(t))
	+\hat\la_it-\sum_{j\in \calJ_i}\hat\mu_{j}T_{j}(t)+\sum_{j\in \calJ_i}\mu_{j}Y_{j}(t),
	\\
	\label{07}
	I_k(t)
	&=\sum_{j\in \calJ^k}Y_{j}(t).
\end{align}
These processes satisfy the constraints
\begin{align}\label{08}
	&
	X_i(t)\geqslant0 \text{ for all $t$},\\
	&
	\text{$I_k$ has sample paths in $C^{0,+}_\R[0,\iy)$},\\
	&
	\text{if $T_{j}$ is flat on some interval $(s,t)$
		then $Y_{j}$ is non-increasing on it}.
	\label{08+}
\end{align}
The BCP consists of minimizing the cost
\[
\E\int_0^\iy e^{-\gamma t}h(X(t))dt,
\]
over all tuples $(T,Y, X,I)$ satisfying \eqref{06}--\eqref{08+}.
A precise definition must also account for
adaptedness w.r.t.\ a filtration, as in the definition of a WCP
given in \S\ref{sec:wcpp};
however, since this problem is presented here only for motivational
purposes, we will not formulate it in full detail.

There are similarities between this BCP and those that appear in
\cite{harlop}, \cite{bw1}, \cite{bw2}. An important difference
is the appearance of the process $T$ as a control process,
as opposed to it being deterministic in the above references.
This has significant consequences already mentioned in the introduction.

\subsubsection{WCP derivation}\label{sec:wcp}

We next show how the WCP, which was introduced in \S\ref{sec:wcpp}, is obtained
directly from the queueing model, and then discuss the role it plays
in this paper.
Let the process that appears in the definition of the cost \eqref{42} be denoted by $\hat H^n$, where
\begin{equation}\label{143}
	\hat H^n_t=h(\hat X^n_t)=h\cdot\hat X^n_t.
\end{equation}
Using \eqref{eq:balancex} and the definitions of 
$\hat \lambda^n, \hat \mu^n$, a straightforward calculation
gives, with $\iota(t)=t$,
\begin{equation}\label{c5}
	\hat W^n=\hat F^n+n^{1/2}\sum_iy^*_i\la_i\iota-n^{1/2}\sum_iy^*_i\sum_{j\in \calJ_i}\mu_{j}T^n_{j},
\end{equation}
where
\begin{equation}
	\label{096} \hat F^n=\sum_iy^*_i(\hat A^n_i-\sum_{j\in \calJ_i}\hat S^n_{j}\circ T^n_{j}
	+\hat\la^n_i\iota-\sum_{j\in \calJ_i}\hat\mu^n_{j}T^n_{j}).
\end{equation}
The analysis of the problem via workload considerations is based
on two key inequalities. The first, which uses the positivity of $y^*$
and the definition of $q$ in \eqref{145}, is the following:
\begin{equation}\label{144}
	\hat H^n_t=\sum_ih_iy^*_i(y^*_i)^{-1}\hat X^n_i(t)\ge h_q(y^*_q)^{-1}\hat W^n_t,
	\thickspace
	\text{with equality when}\thickspace \sum_{i\neq q}\hat X^n_i(t)=0.
\end{equation}
The second key inequality is given by the following lemma. To state it we first
introduce the {\it Skorohod map on the half line},
denoted throughout this paper by
$\Gam:D_\R[0,\iy)\to D_{\R_+}[0,\iy)\times D_\R^+[0,\iy)$. It
takes a function $\psi$ to a pair $(\ph,\eta)$, where
\[
	\ph(t)=\psi(t)+\eta(t),\qquad \eta(t)=\sup_{0\les s\les t}\psi(s)^-,
	\qquad t\ges0.
\]
The corresponding maps $\psi\mapsto\ph$ and $\psi\mapsto\eta$ are denoted
by $\Gam_1$ and $\Gam_2$, respectively.
\begin{lemma}\label{lem5}
One has
\begin{equation}\label{097}
	\hat W^n_t\geqslant \Gam_1[\hat F^n](t),
	\qquad t\ge0.
\end{equation}
\end{lemma}
The proof appears at the end of this section.

Inequalities \eqref{144} and \eqref{097} allow us to bound the cost
$\hat J^n$ defined in \eqref{42} as follows:
\begin{equation}\label{094}
	\hat J^n(T^n)=\E\int_0^\iy e^{-\gamma t}\hat H^n_tdt
	\geqslant
	h_q(y^*_q)^{-1}\E\int_0^\iy e^{-\gamma t}\Gam_1[\hat F^n](t)dt,
	\qquad
	T^n\in\calA^n.
\end{equation}
The derivation that follows is concerned with the asymptotics
of the RHS of the above inequality.

We take formal limits in \eqref{c5}--\eqref{096}.
If we assume that $(\hat A^n,\hat S^n,T^n)\To(A,S,T)$
then by \eqref{096}, $\hat F^n\To F$ defined by
\begin{equation}\label{eq:wtilde}
	F(t)=\sum_iy^*_i\Big(A_i(t)+\hat\la_it
	-\sum_{j\in \calJ_i}S_{j}(T_{j}(t))+\hat\mu_{j}T_{j}(t)\Big).
\end{equation}
Assuming further that $T$ is of the form $\int_0^\cdot\X(t)dt$ where
$\X(t)=(\X_{j}(t))$ takes values in $\slp$,
using the definition \eqref{50} of $b$ and \eqref{eq:wtilde}, we have
\begin{align}
	\label{53}
	Z_t:=\Gam_1[F](t)
	&=\int_0^t\sum_iy^*_i\Big(\hat\la_i-\sum_{j\in \calJ_i}\hat\mu_{j}\X_{j}(s)\Big)ds
	+M_t+L_t
	\\
	\notag
	&=\int_0^t b(\X_s)ds+M_t+L_t,
\end{align}
where $Z_t\geqslant0$, $L$ has sample paths in $C^{0,+}_\R[0,\iy)$,
$\int_{[0,\iy)}Z_tdL_t=0$, and
\begin{equation}\label{166}
	M_t=\sum_iy^*_i\Big(A_i(t)-\sum_{j\in \calJ_i}S_{j}(T_{j}(t))\Big).
\end{equation}
If one can further show that $A_i$ and $S_{j}\circ T_{j} $ are all
martingales on a common filtration, and moreover,
$\langle S_{j}\circ T_{j}\rangle_t=\langle S_{j}\rangle_{T_{j}(t)}$
then $M_t$ is a martingale, and
$$
\langle M\rangle_t=\int_{0}^{t}\sum_i(y^*_i)^2\Big(\sig_{A,i}^2
+\sum_{j\in \calJ_i}\sig_{S,j}^2\X_{j}(s)\Big)ds=\int_{0}^{t}\sigma^2(\X(s))ds,
$$
where the definition \eqref{50} of $\sig$ is used.
Hence
\[
Z_t=\int_0^t b(\X_s)ds+\int_0^t\sigma(\X_t)dB_t+L_t,
\]
for $B$ a SBM. This is precisely \eqref{15} with initial condition $z=0$.
(Although the case $z=0$ reflects our assumption $\hat X^n(0)=0$,
the formulation of \eqref{15} with a general initial condition $z\ge0$
will be important when we appeal to the HJB equation.)
Moreover, because by construction $Z=\Gam_1[F]$,
the cost $J_{\rm WCP}$ is formally the limit of the expected integral
on the RHS of \eqref{094}.
This suggests that the value function for the QCP is asymptotically
bounded below by $h_q(y^*_q)^{-1}V_{\rm WCP}(0)$,
which is the main result of this paper.
Indeed, \eqref{094} is the starting point for its proof in \S\ref{sec5}.

From the above discussion it also follows that in order
for the lower bound to be tight, that is, for
there to exist a QCP policy that asymptotically achieves it,
the two inequalities \eqref{144}, \eqref{097} should asymptotically be achieved
as equalities, and the RHS of \eqref{094}
should be asymptotic to the optimal WCP solution.
For \eqref{144} to hold as near equality,
one must keep $\hat X^n_i$, $i\ne q$ close to zero.
As will become clear from the proof below of Lemma \ref{lem5},
for \eqref{097} to hold as equality up to a negligible term,
the process $\hat L^n$ (more precisely, the sum $\hat L^n+\hat L^n_{\rm AN}$,
where the latter is defined in the proof of Lemma \ref{lem5}), should be kept close
to $\Gam_2[\hat F^n]$. 
For $\hat L^n$  
this entails that, up to negligible terms,
none of the servers should accumulate idleness when there is work in the system, and
for $\hat L^n_{\rm AN}$  this entails that, up to negligible terms, no server should engage in any activity 
$j \in \calA_{\rm AN}$, that is, always non-basic, except possibly when the workload is close to zero.
For PSS satisfying the HTC 
these two points were recognized in \cite{har98}, who showed that, if sufficient care is not taken in keeping the components
$\hat X^n_i$, $i\ne q$ close to zero, the system could be overwhelmingly overloaded,
as well as in \cite{bw1}, who have constructed a threshold policy
that succeeds in achieving the two goals.
When the HTC does not hold there are additional considerations involving
the third aspect mentioned above. Namely, for the RHS of \eqref{094}
to be asymptotic to the WCP solution, certain switching control policies
need to be implemented. This is the subject of \cite{ourupperbound}.

\begin{proof}[Proof of Lemma \ref{lem5}.]
Any activity involves exactly one class and one server. For that reason, the following three operations are equivalent and can be used interchangably, depending on what index is needed in the summation: $\sum_i\sum_{j\in \calJ_i}$, $\sum_k\sum_{j\in \calJ^k}$, $\sum_i\sum_k\mathds1_{(i,k)\in \calJ}$. By Lemma \ref{lem:primaldual} part 4, the last term in \eqref{c5} can be modified using the following identity
\begin{align}\label{c2}\notag
	\sum_iy^*_i\sum_{j\in \calJ_i}\mu_{j}T^n_{j}
	&=\sum_i\sum_k\mathds1_{(i,k)\in \calJ\cap\calA_{\rm PB}}y^*_i\mu_{ik}T^n_{ik}+\sum_i\sum_k\mathds1_{(i,k)\in \calJ\cap\calA_{\rm AN}}y^*_i\mu_{ik}T^n_{ik}\\
	&=\sum_k\sum_{j\in \calJ^k\cap\calA_{\rm PB}}z^*_kT^n_{j}\notag
	+\sum_k\sum_{i}\mathds{1}_{(i,k)\in \calJ\cap\calA_{\rm AN}}y^*_i\mu_{ik}T^n_{ik}\\
	&=\sum_kz^*_k\sum_{j\in \calJ^k}T^n_{j}\notag
	-\sum_k\sum_{i}\mathds{1}_{(i,k)\in \calJ\cap\calA_{\rm AN}}(z^*_k-y^*_i\mu_{ik})T^n_{ik}\\
	&= \sum_kz^*_k\sum_{j\in \calJ^k}T^n_{j} - L^n_{\rm AN},
\end{align}
where the last equality defines the process
$L^n_{\rm AN}$.

In view of Lemma \ref{lem:primaldual} part 5 and the fact that
$T^n_{ik}$ are nonnegative and nonndecreasing, so is $L^n_{\rm AN}$.
Denote $\hat L^n_{\rm AN}=n^{1/2}L^n_{\rm AN}$.
By \eqref{c5}, \eqref{c2} and Lemma \ref{lem:primaldual} part 2,
\begin{equation*}
	\hat W^n=\hat F^n+n^{1/2}\iota-n^{1/2}\sum_kz^*_k\sum_{j\in \calJ^k}T^n_{j} + \hat L^n_{\rm AN}.
\end{equation*}
By \eqref{eq:dualreformulation}, $\sum_k z^*_k=1$,
and by \eqref{41}, $\hat I^n_k+n^{1/2}\sum_{j\in \calJ^k}T^n_{j}=n^{1/2}\iota$.
Hence
\begin{align*}
\hat W^n&=\hat F^n+n^{1/2}\iota-\sum_kz^*_k(n^{1/2}\iota-\hat I^n_k)
+\hat L^n_{\rm AN}
	\\
	&= \hat F^n+z^*\cdot\hat I^n+\hat L^n_{\rm AN}\\
	&= \hat F^n+ \hat L^n+\hat L^n_{\rm AN}.
\end{align*}

It is well known that the Skorohod map $\Gam_1$ possesses the following {\it minimality property}
(see \cite[Section 2]{cheman}).
Let $\psi\in D_\R[0,\iy)$ and assume that
$\ph=\psi+\eta$ where $\eta\in D^+_{\R}[0,\iy)$ and
$\ph(t) \ge 0$  for all $t$. Then $\ph(t)\ges\Gam_1[\psi](t)$ for all $t$.

Because the process $\hat W^n$ has non-negative sample paths and
$\hat L^n$ and $\hat L^n_{\rm AN}$
have sample paths in $D^+_\R[0,\iy)$, it follows from the minimality
property that $\hat W^n_t\ge\Gam_1[\hat F^n](t)$ for all~$t$.
This completes the proof.
\end{proof}

\begin{remark}\label{rem-lan}
We now provide an intuitive explanation for  $\hat L^n_{\rm AN}$.
The process $\hat L^n_{\rm AN}$ involves a sum over the always non-basic activities.
Recall that Lemma \ref{lem:primaldual} part 5 states that, for any activity $j=(i,k) \in \calA_{\rm AN}, y^*_i\mu_{ik} < z^*_k$,
while Lemma \ref{lem:primaldual} part 4 states that for any activity $j=(i,k) \in \calA_{\rm PB}, y^*_i\mu_{ik} = z^*_k$.
Intuitively, with $y^*_i$ interpreted as the average work of a class $i$ customer, $y^*_i\mu_{ik}$ represents the rate at which activity $j$ removes class $i$ work from the system.
Since each server $k$ has at least one activity $j'=(i',k) \in  \calA_{\rm PB}$, for that activity $y^*_{i'}\mu_{i'k}=z^*_k > y^*_i\mu_{ik}$,
so that activity $j'$ can be viewed as more `efficient' than activity $j$.
Thus $\hat L^n_{\rm AN}(t)$ is the additional amount of work that could have been removed from the system by time $t$ if the effort expended on activities in $\calA_{\rm AN}$ had instead been used on activities in $\calA_{\rm PB}$.
\end{remark}

\subsection{Candidate solution to the WCP}\label{sec42}

In the following result, under the assumption
of existence of a classical solution to the HJB equation, a candidate solution
to the WCP is constructed. This result is later used to prove Proposition~\ref{prop0}
which states that a classical solution exists,
and Corollary \ref{cor1} which verifies the optimality of this candidate.

\begin{lemma}\label{lem3}
	\begin{enumerate}
		\item
		For any measurable map $\xi^0:\R_+\to\slp$ and $z\geqslant0$,
		there exists a weak solution to the SDE
		\begin{equation}\label{54}
			Z_t=z+\int_0^tb(\xi^0(Z_s))ds+\int_0^t\sig(\xi^0(Z_s))dB_s+L_t,
		\end{equation}
		where $Z_t\geqslant0$, $L$ has sample paths in $C^{0,+}_\R[0,\iy)$,
		and $\int_{[0,\iy)}Z_tdL_t=0$.
		\item
		Assume that a classical solution $u$ to \eqref{14}--\eqref{14+} exists.
		Then letting
		\begin{align}\label{167}
			\xi^u(z)&=\sum_{m\leqslant M}\one_{A^m}(z)\xi^{*,m}, \quad\text{where}
			\\
			\notag
			A^m&=A^m_u=\{z\in\R_+:\BH(u'(z),u''(z))=\bar\BH(u'(z),u''(z),\xi^{*,m})\}\setminus \cup_{m'<m}A^{m'},
		\end{align}
		there exists a weak solution to \eqref{54} in the special case
		$\xi^0=\xi^u$.
	\end{enumerate}
	
\end{lemma}

In the literature on Markov control problems, a map from the state space to
the control action space is often called a {\it stationary (feedback) control policy},
or a {\it policy} for short. For our WCP, a policy is thus
a bounded measurable map $\xi^0:\R_+\to\slp$. Thus
part 1 of the lemma above shows that a policy $\xi^0$
induces an admissible control system for the WCP, and part 2 provides
a candidate optimal policy.

\begin{proof}
1.  We appeal to Th.\ 5.5.15 and Rem.\ 5.5.19 in
\cite{kar-shr}, according to which a weak solution exists
to the SDE
\begin{equation}\label{55}
	Y_t=y+\int_0^t\tilde b(Y_s)ds+\int_0^t\tilde\sig(Y_s)dB_s
\end{equation}
on the real line,
where $\tilde b,\tilde\sig:\R_+\to\R$ are measurable, $\tilde b$ bounded,
and $\tilde\sig^2$ bounded away from zero and infinity.
In particular, this is true for $\tilde b$ and $\tilde\sig$ given by
$\tilde b(x)=\sgn(x)b(\xi^0(|x|))$, $\tilde\sig(x)=\sgn(x)\sig(\xi^0(|x|))$, $x\in\R$
and the initial condition $y=z\ges0$, where we define $\sgn(x)=1$ if $x>0$
and $\sgn(x)=-1$ if $x\les0$.
These coefficients are indeed bounded because the set $\slp$
is bounded, and moreover $\tilde\sig^2$ is bounded away from zero because
$\sig_m >0, 1 \le m \le M$.
Let then $(Y,B)$ be a corresponding solution to \eqref{55}.
Then by Tanaka's formula (\cite{rev-yor} Theorem VI.1.2, p.\ 222)
there exists a process $L^0$ with sample paths in $C^{0,+}_\R[0,\iy)$ such that
\begin{align*}
	|Y_t|&=
	y+\int_0^t\sgn(Y_s)^2b(\xi^0(|Y_s|))ds
	+\int_0^t\sgn(Y_s)^2\sig(\xi^0(|Y_s|))dB_s+L^0_t,
\end{align*}
showing that $(Z,L^0,B)$ is a solution to \eqref{54},
where $Z=|Y|$.
Proposition VI.1.3 on p.\ 222 of \cite{rev-yor} justifies that
$L^0$ satisfies $\int_{[0,\iy)}Z_tdL^0_t=0$.\vspace{0.2cm}

2. One only needs to show that $A^m$ are Borel measurable.
To this end, note that the set of $(v_1,v_2)$ for which
$\BH(v_1,v_2)=\bar\BH(v_1,v_2,\xi^{*,m})$ is Borel.
Since $u$ is $C^2$, so are the sets used in the recursive
construction of $A^m$, as the inverse images
of the above sets under a continuous map.
\end{proof}

\subsection{Classical solution and WCP optimality}
\label{sec43}

\begin{proof}[Proof of Proposition \ref{prop0}.]
The proof is based on results from the PDE literature
regarding the existence of classical solutions to
\eqref{14} on a bounded domain with Dirichlet boundary conditions.

Step 1: Existence of classical solutions on a bounded domain.
Consider
\begin{equation}
	\label{a1}
	{\mathbb H}(u'(z),u''(z))+z-\gamma u(z)=0,
	\qquad z\in(0,z_0),
\end{equation}
with boundary conditions $u(0)=r_0$, $u(z_0)=r_1$, where $z_0\in(0,\iy)$
and $r_i\in\R$ are constants.
We rely on \cite[Theorem 17.18]{gil-tru}. To apply this result
it suffices to verify conditions
\cite[(17.62)]{gil-tru} (see the discussion on the Bellman equation that follows
the statement of \cite[Theorem 17.18]{gil-tru}).
To this end, note that uniform ellipticity holds in our setting because
$\sig_m>0$ for all $m$. The coefficients denoted $a$, $b$, $c$ and $f$
in \cite{gil-tru} correspond here to $\sig(\xi)^2/2$, $b(\xi)$,
$-\gamma$ (which do not depend on $z$)
and $-z$. Clearly these functions and their derivatives w.r.t.\ $z$ are bounded on
the finite interval. It follows that a classical solution to \eqref{a1} with the Dirichlet
boundary conditions $r_0$, $r_1$
uniquely exists, denoted in what follows by $u^{(z_0,r_0,r_1)}$.

Step 2: A control problem on a bounded domain.
Fix $z_0$ and let $v=u^{(z_0,V_{\rm WCP}(0),V_{\rm WCP}(z_0))}$.
Consider the control problem up to exiting the domain $(0,z_0)$,
defined analogously to \eqref{51} by
\begin{align*}
J^{(z_0)}(z,\frS)&=\E\Big[\int_0^\tau e^{-\gamma t}Z_tdt+e^{-\gamma\tau}V_{\rm WCP}(Z_\tau)\Big],
\\
V^{(z_0)}(z)&=\inf\{J^{(z_0)}(z,\frS):\frS\in\calA_{\rm WCP}(z)\},
\end{align*}
for $z\in[0,z_0]$ and $\tau=\inf\{t:Z_t\in\{0,z_0\}\}$
(on the event $\tau=\iy$,
the term $e^{-\gamma\tau}V_{\rm WCP}(Z_\tau)$ is defined as $0$).
Given Step 1 and the existence of a weak solution to \eqref{54}--\eqref{167}
stated in Lemma \ref{lem3}, an argument similar to the one
provided in the verification theorem
\cite[Th.\ IV.5.1]{fle-son} shows that
$v(z)=V^{(z_0)}(z)$ for $z\in[0,z_0]$. We omit the details.

Step 3: $V_{\rm WCP}$ is a solution to \eqref{14}--\eqref{14+}.
By the dynamic programming principle
(see e.g.\ \cite[Theorem 6, page 132]{kry-book} with $r_t=0$ and $\ph_t=\gamma t$),
we have
$V_{\rm WCP}(z)=V^{(z_0)}(z)$ for all $z\in[0,z_0]$.
Combining with the last step, 
$V_{\rm WCP}(z)=v(z)$ for all $z\in[0,z_0]$.
Since $z_0$ is arbitrary, this shows that $V_{\rm WCP}$
is a classical solution to \eqref{14} on all of $[0,\iy)$.

Next, by the boundedness
of $b$ and $\sig$ in \eqref{15}, it is easy to see
that $\E Z_t\le c(z+t)$ for some constant $c$ which does not depend on the
control system.
Hence by \eqref{51}, we have $0\le V_{\rm WCP}(z)\le c(z+1)$.

As for the boundary condition $V'_{\rm WCP}(0)=0$, it is a standard fact
that this follows from the reflection at the origin. For completeness
we provide a proof. Given $z>0$ and $\eps>0$ let
$\frS=(\Om',\calF',(\calF'_t)_{t\geqslant 0},\PP',B,\X,Z,L)\in\calA_{\rm WCP}(z)$ be such that
\[
J_{\rm WCP}(z,\frS)<V_{\rm WCP}(z)+\eps.
\]
By \eqref{15}, one has $(Z,L)=\Gam(z+K)$, where
$K_t=\int_0^tb(\X_s)ds+\int_0^t\sig(\X_s)dB_s$.
If we define $\tilde\frS$ using the same probability space and
processes $B$ and $\X$, but with $(\tilde Z,\tilde L)=\Gam(K)$,
then $\tilde\frS\in\calA_{\rm WCP}(0)$.
Moreover, $\tilde Z_t\le Z_t$ for all $t$. As a result,
\[
V_{\rm WCP}(0)\le J_{\rm WCP}(0,\tilde\frS)\le J_{\rm WCP}(z,\frS)\le V_{\rm WCP}(z)+\eps.
\]
Since $\eps$ is arbitrary it can be dropped.

Similarly, let
$\frS=(\Om',\calF',(\calF'_t)_{t\geqslant 0},\PP',B,\X,Z,L)\in\calA_{\rm WCP}(0)$ be such that
\[
J_{\rm WCP}(0,\frS)<V_{\rm WCP}(0)+\eps.
\]
Then (with notation as before)
$(Z,L)=\Gam(K)$, and defining $(\tilde Z,\tilde L)=\Gam(z+K)$ one has
$\tilde\frS\in\calA_{\rm WCP}(z)$. Denoting $\Del_t=\tilde Z_t-Z_t$,
\[
V_{\rm WCP}(z)-V_{\rm WCP}(0)\le J_{\rm WCP}(z,\tilde\frS)-J_{\rm WCP}(0,\frS)+\eps\le\E\int_0^\iy e^{-\gamma t}\Del_tdt+\eps.
\]
Let $\tau=\inf\{t:\tilde Z_t=0\}=\inf\{t:K_t=-z\}$.
Then it is easy to see by the definition of $\Gam$ that $\Del_t\le z$ for all $t$,
while $\Del_t=0$ for all $t\ge\tau$.
Hence
\begin{equation}\label{b1}
	0\le V_{\rm WCP}(z)-V_{\rm WCP}(0)\le z\E\int_0^\tau e^{-\gamma t}dt+\eps\le z\E[\tau\w c_0]+\eps,
\end{equation}
where $c_0=\gamma^{-1}$.
For the martingale $\hat K_t=\int_0^t\sig(\X_s)dB_s$ one
has $\lan\hat K\ran_t=\int_0^t\sig^2(\X_s)ds$, and because $\sig$ is bounded away from
$0$ and $\iy$, one has $0<c_1\le (t_2-t_1)^{-1}(\lan\hat K\ran_{t_2}-\lan\hat K\ran_{t_1})\le c_2<\iy$,
and thus the map $t\mapsto\lan\hat K\ran_t$ admits a strictly increasing, Lipschitz inverse $T$.
By \cite[Theorem 3.4.6]{kar-shr}, $Y_t:=\hat K_{T(t)}$ is a SBM.
By the boundedness of $b$, $K_t\le \hat K_t+ct$. Hence for any $t>0$,
\begin{align*}
	\PP'(\tau>t)&=\PP'(K_s>-z \text{ for all } s\le t)
	\\
	&\le \PP'(\hat K_s+cs>-z \text{ for all } s\le t)
	\\
	&=\PP'(Y_s+cT(s)>-z \text{ for all } s\le \lan\hat K\ran_t)
	\\
	&\le\PP'(Y_s+cc_1^{-1}s>-z \text{ for all } s\le c_1t).
\end{align*}
This bound depends only on $t$ and $z$, and converges to zero as $z\to0$.
Hence it follows that $\E[\tau\w c_0]\le\kappa(z)$ for some function $\kappa$
which satisfies $\lim_{z\downarrow0}\kappa(z)=0$.
Using this in \eqref{b1} shows that $z^{-1}|V_{\rm WCP}(z)-V_{\rm WCP}(0)|\le\kappa(z)+z^{-1}\eps$.
Sending $\eps\downarrow0$ and then $z\downarrow0$
shows that $V'_{\rm WCP}(0)$ exists and equals $0$.
This completes the proof that
$V_{\rm WCP}$ it is a classical solution to \eqref{14}--\eqref{14+}.

Step 4: uniqueness on all of $[0,\iy)$.
It remains to show that if $u$ is any classical solution to \eqref{14}--\eqref{14+}
then $u=V_{\rm WCP}$. This is also done
in the spirit of \cite[Theorem IV.5.1]{fle-son}, but here the details
are more subtle due to the boundary condition at infinity,
and we provide the proof.

Let $\frS=(\Om',\calF',(\calF'_t)_{t\geqslant 0},\PP',B,\X,Z,L)$ be an admissible
control system. It will be shown that $J_{\rm WCP}(z,\frS)\ge u(z)$.
Since $u$ is $C^2$,
a use of It\^o's formula gives
\begin{align}
	e^{-\gamma t}u(Z_t)
	&=u(z)+\int_{0}^{t}e^{-\gamma s} u'(Z_s)\left[b(\X_s)d s+\sigma(\X_s)dB_s+dL_s\right]-\gamma\int_{0}^{t}e^{-\gamma s}u(Z_s) ds\notag\\
	&\quad +\frac{1}{2}\int_{0}^{t}e^{-\gamma s}u''(Z_s)\sigma^2(\X(s))ds.\label{eq:verif}
\end{align}
Because $\int_{0}^{\infty}Z_tdL_t=0$ and $u'(0)=0$, we have
\[
\int_{0}^{t}e^{-\gamma s} u'(Z_s)dL_s=\int_0^t e^{-\gamma s}u'(Z_s)\one_{\{Z_s=0\}} dL_s=0.
\]
By equation \eqref{14},
for any member $\xi$ of $\slp$ we have
\begin{equation}
	\frac{1}{2}\sig(\xi)^2u''(y)+b(\xi)u'(y)+y-\gamma u(y)\geqslant 0,
	\qquad y\in\R_+.
	\label{eq:hjbineq}
\end{equation}
Substituting in \eqref{eq:verif}, denoting
$\tilde M_t=\int_{0}^{t}e^{-\gamma s}u'(Z_s)\sigma(\X_s)dB_s$,
\begin{equation}\label{eq:ineqito}
	e^{-\gamma t}u(Z_t)\geqslant u(z)+\tilde M_t-\int_{0}^{t}e^{-\gamma s }Z_s ds.
\end{equation}
We argue that $\tilde M_t$ is in $L^1$ for any $t$,
by which the local martingale $\tilde M_t$ is a martingale.
First we show that $\tilde M_t^+$ is in $L^1$ for any $t$.
In view of \eqref{eq:ineqito}, to bound $\E[\tilde M_t^+]$, it suffices to estimate
$\E[u(Z_t)]$ and $\E[\int e^{-\gamma s}Z_sds]$.
To bound $Z_t$, fix $t\geqslant 0$ and denote $\tau=\sup\{s\in[0,t]:Z_s=0\}$.
We distinguish two cases. Case~(i): $\tau\in[0,t]$ and $Z_\tau=0$.
In this case
$$
Z_t=\int_{\tau}^{t}b(\X_s)ds +\int_{\tau}^{t}\sigma(\X_s)dB_s,
$$
because
$L_t=L_\tau$. Case (ii): $Z_s>0$ for all $s\in[0,t]$. Then
$$Z_t=z+\int_{0}^{t}b(\X_s)ds +\int_{0}^{t}\sigma(\X_s)dB_s.$$
Since $Z_t\geqslant0$, in both cases one has
$$
\|Z\|_t\les c+ct+2\sup_{\theta\leqslant t}\Big\lvert\int_{0}^{\theta}\sigma(\X_s)dB_s\Big\rvert
$$
for a constant $c$ that does not depend on $t$.
Because $\sig$ is bounded, the second moment
of the last term is bounded by $ct$, which gives $\E\|Z\|_t\le c(1+t)$.
In particular, $\E[\int_0^t e^{-\gamma s}Z_sds]<\iy$. Moreover,
by \eqref{14+},
\begin{equation}\label{eq:supu}
	\E[\|u(Z_.)\|_t]\leqslant c(1+t).
\end{equation}
In view of \eqref{eq:ineqito}, this shows not only that
$\E[\tilde M_t^+]<\iy$ but also $\E[\|\tilde M^+\|_t]<\iy$.
Next, let $\tau^N$ be a sequence of stopping times
such that $\tau^N\to +\infty$ a.s., and $\tilde M_{t\wedge \tau^N}$ is a martingale for every $N$.
Then $\E[\tilde M^+_{t\wedge \tau^N}]=\E[\tilde M^-_{t\wedge \tau^N}]$.
Using this identity and Fatou's Lemma,
\[
\E[\tilde M^-_t]\leqslant \liminf_{N\to +\infty}\E[\tilde M^-_{t\wedge \tau^N}]
=\liminf_{N\to +\infty}\E[\tilde M^+_{t\wedge \tau^N}]
\leqslant
\E[\|\tilde M^+\|_t]<\iy.
\]
This shows that $\tilde M_t$ is a martingale.

Taking the mean and rearranging in \eqref{eq:ineqito}, for any $t$,
\begin{equation}\label{56}
	J_{\rm WCP}(z,\frS)\geqslant
	\E\int_{0}^{t}e^{-\gamma s} Z_s ds
	\geqslant u(z)-e^{-\gamma t}\E\left[u(Z_t)\right].
\end{equation}
In view of \eqref{eq:supu}, letting $t\to +\infty$ gives  $J_\text{WCP}(z,\frS)\geqslant u(z)$.

To show that $u=V_{\rm WCP}$ it remains to find $\frS$ for which
$J_\text{WCP}(z,\frS)= u(z)$.
To this end, let $\xi^u$ be as in \eqref{167}.
Let an initial condition $z$ be given.
Form a control system
$\frS^u=(\Om,\calF,\{\calF_t\},\PP,\X,Z,L)$ by letting
$(Z,B,L)$ be a tuple satisfying \eqref{54} with $\xi^0=\xi^u$ and initial condition $z$,
$\X_t=\xi^u(Z_t)$ and $\calF_t=\sig\{(Z_s,L_s,B_s):s\in[0,t]\}$, $t\ge0$.
Here we have assumed, without loss of generality,
that the weak solution to \eqref{54} is defined on our original probability
space. Then using $\X_t=\xi^u(Z_t)$ in \eqref{14} gives
\[
b(\X_s)u'(Z_s)+\frac{1}{2}\sig^2(\X_s)u''(Z_s)+Z_s-\gamma u(Z_s)=0.
\]
Hence by \eqref{eq:verif}, \eqref{eq:ineqito} holds with equality. Thus
$\E\int_0^te^{-\gamma s}Z_sds=u(z)-e^{-\gamma t}\E[u(Z_t)]$.
To take the $t\to\iy$ limit,
use monotone convergence on the LHS and the bound \eqref{eq:supu} on the RHS.
This shows $J_{\rm WCP}(z,\frS^u)=u(z)$ and completes the proof.
\end{proof}

As a corollary of the proposition and its proof we obtain
a construction of an optimal solution to the WCP.

\begin{corollary}\label{cor1}
	Denote $u=V_{\rm WCP}$. Let $\xi^u$ be as in \eqref{167}.
	Let an initial condition $z$ be given.
	Form an admissible control system
	$\frS^u=(\Om,\calF,\{\calF_t\},\PP,\X,Z,L)$ by letting
	$(Z,B,L)$ be a tuple satisfying \eqref{54} with $\xi^0=\xi^u$ and initial condition $z$,
	$\X_t=\xi^u(Z_t)$ and $\calF_t=\sig\{(Z_s,L_s,B_s):s\in[0,t]\}$, $t\ge0$.
	Then $\frS^u$ is optimal for the WCP with initial condition $z$.
\end{corollary}

The next result is concerned with a special case in which there exists a dominant mode,
showing that the WCP's solution is to always use this mode.

\begin{proposition}\label{prop2}
	Assume there exists a mode, say, $\xi^{*,1}$, such that for all $2\le m\le M$
	one has $b(\xi^{*,1})\le b(\xi^{*,m})$ and $\sig(\xi^{*,1})\le\sig(\xi^{*,m})$.
	Then it is optimal to always use this mode. In other words,
	the statement of Corollary \ref{cor1} holds with $\xi^0(z)=\xi^u(z)=\xi^{*,1}$, $z\in\R_+$.
\end{proposition}

\begin{proof}
In view of Proposition \ref{prop0} and Corollary \ref{cor1},
it suffices to find a classical solution $u$ to
\eqref{14}--\eqref{14+} which satisfies $A^1_u=\R_+$,
where we use the notation of Lemma \ref{lem3}.
We rely on the explicit expression for the cost associated with the policy
that always uses the mode $\xi^{*,1}$, (see equation (46) on p.\ 104 of \cite{Sheng78}).
This function is given by
\[
u(z)=c_1+c_2z+\frac{c_2}{c_3}e^{-c_3z},
\]
with explicit
constants $c_i\in(0,\iy)$. In particular, $u'(z)\ge0$ and $u''(z)\ge0$ for all $z$,
and $\bar\BH(u'(z),u''(z),\xi^{*,1})+z-\gamma u(z)=0$ for all $z$.
Using the assumption on $b(\xi^{*,1})$ and $\sig(\xi^{*,1})$ and the
positivity of $u'$ and $u''$ shows that
\[
\bar\BH(u'(z),u''(z),\xi^{*,1})=\min_{m\le M}\bar\BH(u'(z),u''(z),\xi^{*,m})
=\BH(u'(z),u''(z)).
\]
As a result, $u$ is a solution to \eqref{14}. Clearly, it also satisfies the boundary conditions
\eqref{14+}. This shows that $A^1_u$ of \eqref{167} is all of $\R_+$, with $u$
a classical solution to \eqref{14}--\eqref{14+}.
\end{proof}

When it is optimal to use only one mode, one can readily adapt results that were proved when the solution to the LP is unique. The policy introduced in \cite{bw2} Section 5 is based on a solution to \eqref{02}, as long as the graph of basic activities induced by the solution is a tree. For that reason, we will speak of the ``Bell and Williams policy based on $\xi$'' to describe a policy constructed in the same way as in \cite{bw2} Section 5 as if $\xi$ was the unique solution to \eqref{02}. Note that $i^*$ from their paper is here denoted $q$ and we use $\xi^{*,1}$ as $x^*$.
\begin{corollary}\label{cor2}
	Keep the assumption from Proposition \ref{prop2}. Assume further that $\xi^{*,1}$ is non-degenerate, and let Assumption 3.2 from \cite{bw2} for the arrival and service distributions hold.
Then the Bell and Williams policy based on $\xi^{*,1}$ is AO (the lower bound from Theorem \ref{thm:lowerbound} is tight).
\end{corollary} 
We give here proof based on the results and proofs from \cite{bw2}.
\begin{proof}
Under condition 2. of Assumption \ref{ass:lp}, as argued in \cite{harlop} Section 2, any non degenerate extreme point of $\slp$ has $I+K-1$ basic activities. Construct a restricted PSS by removing all activities associated with non-basic activities of $\xi^{*,1}$.  
Then $\xi^{*,1}$ is the unique solution to \eqref{02} for the restricted PSS. The solution to the original dual is also a solution of the dual for the restricted PSS. With our setup, Assumptions 3.1, 3.2, 3.3, and 3.6 from \cite{bw2} hold in the restricted PSS. Since Theorem 4.3(iii)  of \cite{bw2} holds, the solution to the dual in the restricted PSS is also unique and Assumption 4.4 from \cite{bw2} holds.
Thus the restricted PSS continues to satisfy our Assumption \ref{ass:lp}, and the graph of basic activities  is a tree.

Under the assumption of Corollary \ref{cor1}, $V_{\rm WCP}$ for both the original and restricted PSS is given by 
\[\E\int_{[0,\iy)}e^{-\gamma t}Z_t dt,\]
with 
\[Z_t=z+tb(\xi^{*,1})+\sig(\xi^{*,1})B_t+L_t.\]
This is also the lower bound in Bell and Williams \cite{bw2}.
The results in \cite{bw2} state that their policy is AO for the restricted PSS, so it must also be optimal for the original problem since it achieves the best possible cost.
\end{proof}

{\sc Example (A) - continued.}
We go back to Example (A), now considering it with second order data.
Assume that the squared coefficients of variation $C^2_{A_i}$ and $C^2_{S_j}$ are all
$1$ except $C^2_{S_1}=4$. Set both $\hat\la_i$ to $0$.
As for $\hat\mu_j$, consider two cases. In case (A1), all $\hat\mu_j$ are set to $0$.
In case (A2), they are all $0$ except $\hat\mu_1=1$.

For each of the modes we can compute the drift and squared diffusion coefficients
via \eqref{50}--\eqref{eq:param}.
In case (A1) we obtain
\[
b_1=0,
\quad
\sig_1^2=\frac{3}{7},
\quad
b_2=0,
\quad
\sig_2^2=\frac{15}{49}.
\]
In case (A2),
\[
b_1=-\frac{1}{7},
\quad
\sig_1^2=\frac{3}{7},
\quad
b_2=-\frac{1}{21},
\quad
\sig_2^2=\frac{15}{49}.
\]
It is seen that in case (A1), $b_2\le b_1$ and $\sig_2\le\sig_1$.
By Proposition \ref{prop2}, it is optimal to always use mode $2$. Moreover,
Corollary \ref{cor2} applies with mode $2$.

In case (A2), $b_1<b_2$ and $\sig_1>\sig_2$, and so none of the modes dominates.
As already mentioned, in this case the HJB equation has been solved in
\cite{Sheng78}, showing that both modes must be used. In particular,
there exists a switching point $z^*$ such that it is optimal to use
one mode when the state process $Z$ is below $z^*$,
and the other when it is above. In the case considered here,
mode $2$ is used below $z^*$ and mode $1$ above it. In particular,
the simplification provided by Corollary \ref{cor2} does not apply. Instead,
a certain switching mechanism
which mimics the WCP solution is required for the PSS model itself in order
to attain AO. We refer to \cite{ourupperbound} for more information.

\skp

{\sc Examples (C) and (E) - continued.}
Recall that in Example (C),
because of the LP degeneracy, the results of  \cite{ourupperbound}
do not apply. However, there is a case where we can still state
an AO result. This is the case where mode 2 (the non-degenerate mode) is dominant in the sense
of Proposition \ref{prop2}, i.e., $b(\xi^{*,2})\le b(\xi^{*,1})$ and $\sig(\xi^{*,2})\le \sig(\xi^{*,1})$.
In this case it is optimal to only use this mode, and we can apply Corollary \ref{cor2}.
In other words, under the assumptions of Proposition \ref{prop2},
the Bell and Williams policy based on $\xi^{*,2}$ is in this case AO.
A similar remark holds for Example (E) with mode $\xi^{*,3}$,
which is the only non-degenerate mode in this case.

\section{An asymptotic lower bound}\label{sec5}
\beginsec

The goal of this section is to prove Theorem \ref{thm:lowerbound}.

\subsection{Lemmas and proof of Theorem \ref{thm:lowerbound}}
\label{sec51}

Theorem \ref{thm:lowerbound} will be proved once it is shown that
for every sequence $\{T^n:n\in\N\}$, with $T^n\in\calA^n$,
one has $\liminf_n\hat J^n(T^n)\ge V_0$. A sufficient condition is that
for every subsequence along which $\limsup_{n}\hat J^n(T^n)<\iy$,
one has $\liminf_n\hat J^n(T^n)\ge V_0$.
Let ${\rm Lip}(\slp)$ denote the collection of paths $\theta\in C(\R_+:\R^\calJ)$
satisfying $(t-s)^{-1}(\theta_t-\theta_s)\in\slp$ for all $0\les s<t$, and $\theta_0=0$.
A preliminary lemma, proved in \S\ref{sec52}, is as follows.

\begin{lemma}\label{lem11}
	Let $\{T^n\}$, $T^n\in\calA^n$ be a sequence of admissible controls
	for the QCP and consider a subsequence (relabelled $\{T^n\}$) along which
	$\limsup_{n\to +\infty}\hat J^n(T^n)<+\infty$.
	Then the sequence of processes $T^n$ is $C$-tight. Consider a further subsequence
	along which $(\hat  A^n,\hat S^n,T^n)\Rightarrow(A,S,T)$.
	Then $T$ has sample paths in ${\rm Lip}(\slp)$ a.s., and if $\{\calF'_t\}_{t\geqslant 0}$ is
	a filtration  to which $T$ is adapted then $T$ is given by
	$\int_0^\cdot\X_sds$ for some $\X\in\frX(\{\calF'_t\})$.
\end{lemma}

Theorem \ref{thm:lowerbound} will follow once it is established that
for every subsequence as in Lemma \ref{lem11} there exists
an admissible system for the WCP, $\frS$, such that
\begin{equation}\label{153}
	\liminf_n\hat J^n(T^n)\ges h_q(y^*_q)^{-1}J_{\rm WCP}(\frS).
\end{equation}
We now outline how
$\frS$ is constructed from the tuple $(A,S,T)$.
The tuple and subsequence
along which $(\hat A^n,\hat S^n,T^n)\To(A,S,T)$
are fixed in the rest of this section.

The admissible system can be constructed, without loss of generality, on
the original space $(\Om,\calF,\PP)$. The remaining elements
of $\frS$ to be constructed are a filtration $\{\calH_t\}_{t\geqslant 0}$ and processes $(B,Z,L,\X)$ satisfying
the definition of an admissible system.
The pair $(Z,L)$ is not obtained simply as a  limit of $(\hat W^n,\hat L^n)$,
because for a general sequence of policies for the QCP
it is not true that such a limit satisfies $\int ZdL=0$.
Instead, one takes the limit in \eqref{096} and obtains
$F$ defined in \eqref{eq:wtilde} as the limit of $\hat F^n$,  and sets
$(Z,L)=\Gam[F]$.
Later, when it is established that $\frS$ is an admissible system,
the lower bound \eqref{153} is obtained upon taking limits in
the inequality \eqref{094}.

If one can construct a filtration $\{\calH_t\}_{t\geqslant 0}$
to which $(T,F)$ (and consequently $(Z,L)$) is adapted and on which the processes $A_i$
and $D_{j}\coloneqq S_{j}\circ T_{j}$ are martingales
(compare with the definition of $D^n$ in \eqref{40-}),
then by Lemma \ref{lem11} there exists $\X\in\frX(\{\calH_t\}_{t\geqslant 0})$,
and moreover $M$ defined in \eqref{166}
is also a martingale. Letting $B=\int_0^\cdot\sig(\X_s)^{-1}dM_s$,
the calculations in \S \ref{sec:wcp} show
that $B$ is a SBM and a martingale w.r.t.\ this filtration,
and the construction of $\frS$ is complete.
(The precise details appear in Lemma \ref{lem4}.)

However, the task of constructing the filtration to which $(T,F)$
are adapted and on which $D_{j}$ are martingales
is surprisingly more difficult than when the limiting process $T$
is deterministic (which is the case in all earlier work on
heavy traffic asymptotic control problems). When $T$ is deterministic
there is no question of it being adapted, and moreover
$D_{j}$ are themselves
independent BMs. Thus on the filtration generated by the mutually independent
$A_i,D_{j}$, it is immediate that these processes are martingales, and so is $M$.
Obviously, this direct approach is not fruitful when $T$ is random.

The approach we develop is based on multi-parameter filtrations,
specifically, on a theorem due to Kurtz \cite{kurtz80}(presented as Theorem \ref{th:kurtz} below)
that translates martingales on a multi-parameter filtration to ones on
a single-parameter filtration.
(Multi-parameter filtrations were used in the
related papers \cite{bw1} and \cite{bw2}
for other purposes; see e.g.\ \S 7 of \cite{bw1}).
In preparation to use this tool for our needs we proceed in two main steps.

The first step is carried out in Lemma \ref{lem03}, where it is argued
that, for every $t$ and $n$, a certain transformation of the RV $T^n_t$
(see \eqref{162}) constitutes
a multi-parameter stopping time on the multi-parameter filtration generated
by $\Ups$ and the $I+J$ processes $A^n_i$, $S^n_{j}$ (specifically, \eqref{161}).
This formal statement expresses the intuitively apparent
fact that in order to determine
whether the cumulative busyness until time $t$ in the activities satisfy
$T^n_{j}(t)\le u_{j}$, it is enough to know the collection
$S^n_{j}(v_{j})$ for $v_{j}\le u_{j}$ and $A^n_i(s)$ for $s\le t$
(in addition to $\Ups$).

It will simplify notation if we use, in this section only,
a common notation for the processes $A^n$ and $S^n$,
and, similarly, for $A$ and $S$.
Recall that activities are indexed by $\calJ\subset\calI\times\calK$.
Introduce $\calJ^0\coloneqq \left\lbrace (i,0), i\in \calI\right\rbrace \cup \calJ$ and set $S^n_{i0}=A^n_i$ for $i\in \calI$.
Similarly, $S_{i0}=A_i$, $i\in \calI$.
This way we have a single index set, $\calJ^0$, where
the indices $j=(i,0)$ are associated with the arrival processes
and $j=(i,k)\in\calJ$, as before, with potential service processes.  

The second main step, performed in Lemma \ref{lem:bcp2wcp}(2),
is showing that the $I+J$-dimensional limit process $(A_i,S_{j})$ in our original notation, and $(S_j)$ in our notation for this section,
constitutes a multi-parameter
martingale on a certain filtration defined in terms of the limit processes
$(S, T)$ (specifically, \eqref{155}). This is crucially based on the first step.
Once this step is established, Theorem \ref{th:kurtz} can be applied,
and in several further steps, namely Lemma \ref{lem:bcp2wcp}(3--5)
and Lemma \ref{lem4}, the construction of $\frS$ is
completed by the approach outlined above.

Let  $\calU=\R_+^{\calJ^0}$.
In this section, the letter $u$ is used as a generic member of $\calU$
(rather than a solution to \eqref{14}).
Given $u\in\calU$, let $S^n(u)$ denote $(S^n_{j}(u_{j}))$.
Define the partial order $\LE$
on $\calU$ by $u\LE v$ iff $u_{j}\les v_{j}$ for all $j\in\calJ^0$. We refer to \cite{kurtz80} for more details on multi-parameter filtrations, martingales and stoping times
(with the partial order given above).
The following is Theorem 6.3(a) in \cite{kurtz80}.
\begin{theorem}[Kurtz]
	\label{th:kurtz}
	Let $\calG_u$ be a filtration indexed by $\calU$,
	and let $\tilde S_j$, $j\in\calJ^0$ be a finite collection of independent SBM for which
	\begin{equation}\label{154}
		\Ph_\theta^{\calM}(u)=\prod_{j\in\calM}\exp\{\mathbf{i}\theta_j \tilde S_j(u_j)
		+\textstyle\frac{1}{2}\theta_j^2u_j\}
	\end{equation}
	is a $\{ \calG_u\}$-martingale for all $\calM\subset\calJ^0$.
	Let $ \tilde T(t)$, $t\in[0,\iy)$ be $\{\calG_u\}$-stopping times
	such that $\tilde T$ has sample paths in $C_\R^+[0,\infty)^{I+J}$
	a.s. Then, for each $\calM\subset\calJ^0$, $\prod_{j\in\calM} \tilde S_j( \tilde T_j(t))$ is a continuous local martingale
	with respect to $\calH_t:=\calG_{ \tilde T(t)}$.
\end{theorem}

For $n\in\N$, denote by $\{\calG^n_u\}_{u\in\calU}$ the multi-parameter filtration
\begin{equation}\label{161}
	\calG^n_u=\sigma\left\{\Ups, \, S^n(v)\, :v\LE u\right\},
	\qquad u\in\calU.
\end{equation}
It will further be convenient to
denote, for $T^n\in\calA^n$, by $\bar T^n$ the $I+J$ component process
\begin{equation}\label{162}
	\bar T^n(t)=(t,\ldots,t,T^n_t)=(\bar T^n_{j}(t):j\in\calJ^0), \quad\text{where }
	\bar T^n_{i0}(t):=t,\, i\in \calI.
\end{equation}
A similar notation will be used for an $I+J$ component process $\bar T$, where 
$\bar T_t=(t,\ldots, t,T_t)$ based on the process $T$.
Denote $D_{j}=S_{j}\circ\bar T_{j}$, for all $j\in\calJ^0$ and not just those in $
\calJ$.
The proofs of the next two lemmas appear in \S\ref{sec52}.

\begin{lemma}\label{lem03}
	Fix $n$ and $T^n\in\calA^n$. Then for every $t\ges0$,
	$\bar T^n_t$ is a $\{\calG^n_u\}_{u\in \R_+^{I+J}}$-stopping time.
\end{lemma}

Next we define a multi-parameter filtration based on the processes
$(S,T)$.
Denote by $\mathbf{e}\in\calU$ the vector whose all components are $1$. 
For $u\in\calU$,
let $\del^*=\del^*(u)=u+\mathbf{e}$ and define the  function
$\psi^{(u)}:\calU\to\calU$ as
\begin{equation}\label{158}
	\psi^{(u)}(v)=\begin{cases} v, & v\LE u, \\ \del^*, & \text{otherwise.}
	\end{cases}
\end{equation}
The choice of the constant $\del^*$ is immaterial as long as it lies outside
$\{v:v\LE u\}$.
Let
\begin{equation}\label{155}
	\mathcal{G}_u\coloneqq \sigma\{S(v),
	\, \psi^{(u)}(\bar T_t):\, v\LE u,\,t\geqslant 0\},
	\qquad u\in\calU.
\end{equation}
According to the definition of $\psi^{(u)}$, the $\sig$-field $\calG_u$
carries the information of whether or not $\bar T_t\LE u$,
as well as the full information about $\bar T_t$ provided it is $\LE u$.
It is immediate that, for every $t$, $\bar T_t$ is a stopping time on this filtration.
Hence
\[
\calH_t=\calG_{\bar T_t}, \quad t\ges0,
\]
is a well-defined (single parameter) filtration. 

The second statement in the result below, that allows us to use
Theorem \ref{th:kurtz}, is crucially based on Lemma \ref{lem03}.
Statements 3--5 then follow mostly from Theorem \ref{th:kurtz}.
Denote by $S^\text{unit}=(S^\text{unit}_{j},j\in\calJ^0)$ the SBM (of unit diffusivity)
$\sig_{j}^{-1}S_{j}$, where $\sig_{j}=\sig_{A,i}$ for $j=(i,0)$, and $\sig_{j}=\sig_{S,j}$ for $j\in \calJ$.

\begin{lemma}\label{lem:bcp2wcp}
	\begin{itemize}
		\item[1.]
		$\bar T_t$ (equivalently, $T_t$) is $\calH_t$-adapted.
		\item[2.] For every $u,v\in\calU$, $\calG_u\indep S(u+v)-S(u)$.
		Consequently, $\{S(u)\}$ is a multi-parameter martingale on $\{G_u\}$, and
		the tuple $(\{\calG_u\},\{S^\text{\rm unit}_j\},\bar T_t)$ satisfies the hypotheses
		of Theorem~\ref{th:kurtz}.
		\item[3.] The processes $D_{j}$, $j\in\calJ^0$
		(i.e., $A_i$ and $S_{j}\circ T_{j}$, $i\in \calI$, $j\in \calJ$)
		are $\mathcal{H}_t$-martingales.
		\item[4.] For distinct indices $j;j'\in\calJ^0$ one has
		$\langle D_{j}, D_{j'}\rangle=0$.
		\item[5.] Finally, for any $j\in\calJ^0$,
		$\langle D_{j}\rangle_t=\sigma_{j}^2\bar T_{j}(t)$.
	\end{itemize} 
	
\end{lemma}

We now arrive at the main lemma asserting that an admissible system $\frS$ exists
with the desired properties listed earlier. The proof, presented below,
is based on Lemmas \ref{lem11}--\ref{lem:bcp2wcp}.
The proof of the Theorem \ref{thm:lowerbound} is presented immediately afterwards,
and establishes the fact that $\frS$ indeed satisfies \eqref{153}.
Recall that $\hat F^n$ and $F$
are defined in \eqref{096} and \eqref{eq:wtilde}.

\begin{lemma}\label{lem4}
	Let the assumptions of Lemma \ref{lem11} hold and
	let $\{\calH_t\}$ be the filtration constructed above.
	Then along the aforementioned convergent subsequence,
	where $(\hat A^n,\hat S^n,T^n)\To(A,S,T)$,
	one has
	$(\hat A^n,\hat S^n,T^n,\hat F^n)\To
	(A,S,T,F)$.
	Moreover, there exist on $(\Om,\calF)$
	processes $(B,Z,L)$ such that the tuple
	$\frS:=(\Om,\calF,\{\calH_t\}_{t\geqslant 0},\PP,B,\X,Z,L)$
	forms an admissible control system for the WCP,
	$T=\int_0^\cdot\X_sds$, and
	\begin{align}\label{098}
		(Z,L)&=\Gam[F],
		\notag
		\\
		F_t &=\sum_iy^*_i\Big(A_i(t)+\hat\la_it
		-\sum_{j\in \calJ_i}S_{j}(T_{j}(t))+\hat\mu_{j}T_{j}(t)\Big)\notag\\
		&= \int_0^tb(\X_s)ds+\int_0^t\sig(\X_s)dB_s.
	\end{align}
\end{lemma}

\begin{remark}\label{rem5}
	Note that it is not claimed that $\hat W^n$ converges (or is tight), nor that if it has
	a limit $W$ then $W=Z$. Both statements are false in general.
	The most that can be said at this point is that, in view of \eqref{097},
	one has $Z_t\le W_t$ for all $t$ whenever $W$ is a limit of $\hat W^n$.
\end{remark}

\begin{proof}
Since we have $(\hat A^n,\hat S^n,T^n)\To(A,S,T)$,
it follows by the lemma on random change of time in
\cite[page 151]{bill} that $\hat S^n_{j}\circ T^n_{j}\To S_{j}\circ T_{j}$,
justifying the convergence 
$\hat F^n\To F$, first assumed  above \eqref{eq:wtilde}.

For the construction of $\X$, note that
by Lemma \ref{lem:bcp2wcp}, $T_t$ is $\calH_t$ adapted. Hence by
Lemma~\ref{lem11}, there exists $\X\in\frX(\{\calH_t\})$ such that $T=\int_0^\cdot\X_sds$.

Next, $B$ is constructed. By \eqref{eq:wtilde}, \eqref{166}
and the definition of $b$ in \eqref{50},
$$
F_t=\int_0^tb(\X_s)ds+M_t,
\qquad
M_t=\sum_i y^*_i\Big(A_i(t)-\sum_{j\in\calJ_i}S_{j}(T_{j})\Big).
$$
By Lemma \ref{lem:bcp2wcp}, $M$ is an $\calH_t$-martingale, and
$$\langle M\rangle_t=\int_{0}^{t}\sum_i(y^*_i)^2\Big(\sig_{A,i}^2+\sum_{j\in\calJ_i}\sig_{S,j}^2\X_{j}(s)\Big)ds=\int_{0}^{t}\sigma^2(\X(s))ds,
$$
where the last equality uses the definition of $\sig$ in \eqref{50}.
Recalling that $\sigma$ is bounded away from zero, let $B$
be the $\calH_t$-local martingale defined as
$B_t=\int_{0}^{t}\sigma(\X(s))^{-1}dM_s$.
To compute $\langle B\rangle_t$,
we use Proposition IV.2.7 of \cite{rev-yor} (once with $K=\sig(\X(\cdot))^{-1}$,
$M$ as above and $K\cdot M$ for $N$, and a second time with $M$ for $N$),
giving
$$\langle B\rangle_t=\int_{0}^{t}\dfrac{\sigma^2(\X(s))}{\sigma^2(\X(s))}ds=t.$$
By Levy's characterization (Theorem IV.3.6 of \cite{rev-yor}), $B$ is a
SBM.  Moreover, we have the identity
$M_t=\int_{0}^{t}\sigma(\X(s))dB_s$ (by Proposition IV.2.4 \cite{rev-yor}).
By construction, $F$ and $B$ are $\calH_t$-adapted.
Setting $(Z,L)=\Gam[F]$ we thus obtain \eqref{098} as well as adaptedness
of this pair.
This shows that $\frS$ is admissible, completing the proof.
\end{proof}

\begin{proof} [Proof of Theorem \ref{thm:lowerbound}.]
The result will follow once we show that
for every subsequence as in Lemma \ref{lem11}, one has
\begin{equation}\label{092}
	\liminf_n \hat J^n(T^n)\ges h_q(y^*_q)^{-1}J_{\rm WCP}(\frS),
\end{equation}
where $\frS$ is the corresponding admissible system from Lemma \ref{lem4}.
By \eqref{144} and \eqref{097}, for any finite $t_0$,
\[
\hat J^n(T^n)\ges \E\int_0^{t_0}e^{-\gamma t}h(\hat X^n(t))dt
\ges h_q(y^*_q)^{-1}\E\int_0^{t_0}e^{-\gamma t}\Gam_1[\hat F^n](t)dt.
\]
By the weak convergence $\hat F^n\To F$ stated in Lemma \ref{lem4},
and the continuity of $\Gam_1$ (as a map from $D_\R[0,\iy)$
to itself, equipped with the topology of convergence u.o.c.),
it follows that
$\Gam_1[\hat F^n]\To\Gam_1[F]=Z$, where $Z$ is a member of
the tuple $\frS$ of Lemma \ref{lem4}, and we have used \eqref{098}.
Because the convergence is u.o.c.\ and the time interval is finite,
Fatou's lemma is applicable, giving
\begin{equation}\label{099}
	\liminf_n\hat J^n(T^n)
	\ges h_q(y^*_q)^{-1}\E\int_0^{t_0}e^{-\gamma t}Z_tdt.
\end{equation}
By \eqref{51}, as $t_0\to\iy$, the above expectation converges to $J_\text{WCP}(\frS)$.
This gives \eqref{092} and completes the proof.
\end{proof}

\subsection{Proof of lemmas}
\label{sec52}

\begin{proof}[Proof of Lemma \ref{lem11}.]
The $C$-tightness of the processes $T^n$ is immediate from the fact that their
sample paths are $1$-Lipschitz. Next, assume $\limsup \hat J^n(T^n)<\iy$.
Define the processes $\bar{A}^n, \bar{S}^n$ and $\bar{X}^n$ by
\[
\bar{A}^n(t)=\dfrac{\check A(\lambda^n t)}{n},
\qquad
\bar{S}^n(t)=\dfrac{\check S(\mu^n t)}{n},
\qquad
\bar{X}^n(t)=\dfrac{X^n(t)}{n}.
\]
By the LLN for renewal processes (see \cite[Sec.\ 14]{bill}),
$(\bar{A}^n,\bar{S}^n)\to (\bar A,\bar S)$ in probability
uniformly over compact time sets, where $\bar A(t)=\la t$ and $\bar S(t)=\mu t$.
Consider a subsequence along which $T^n\To T$
and note that $T$ has $1$-Lipschitz sample paths a.s. Then using \eqref{40-}, \eqref{40}
and denoting $\bar X_i(t)=\bar A_i(t)-\sum_{j\in\calJ_i}\bar S_{j}(T_{j}(t))$, we have
$(\bar{A}^n, \bar{S}^n, T^n,\bar X^n)\To(\bar A,\bar S,T,\bar X)$.
First, it is shown that a.s., $\sup_{s<t}\Del(s,t)\in\slp$,
where $\Del(s,t)=(t-s)^{-1}(T_t-T_s)$.

To this end, note that if
$\xi\in\R_+^{\calJ}$ satisfies $\sum_{j\in\calJ_i}\xi_{j}\mu_{j}=\la_i$
for all $i$ and $\sum_{j\in\calJ^k}\xi_{j}\les1$ for all $k\in\calK$,
then it satisfies the conditions \eqref{02} with $\rho=1$,
thus in view of the EHTC $\rho^*=1$, $(\xi,1)$ must be a solution to the LP.
In other words, $\xi\in\slp$.

To apply this to $\Del(s,t)$, note first that
by \eqref{42} and the linearity of $h(\cdot)$,
\[
n^{-1/2}\hat J^n(T^n)=\E\int_0^\iy e^{-\gamma t}h(\bar X^n(t))dt.
\]
By assumption, the LHS converges to zero. Hence by Fatou's lemma,
$\E\int_0^{t_0}h(\bar X(t))dt=0$ for any $t_0>0$.
By the a.s.\ continuity of the sample paths of $\bar X$, the structure
$h(x)=h\cdot x$, $h_i>0$, and the non-negativity of $\bar X_i(t)$,
it follows that $\bar X_i(t)=0$ for all $t$ a.s.
Thus for all $t$,
\[
\lambda_it-\sum_{j\in\calJ_i} \mu_{j}T_{j}(t)=0.
\]
This shows $\sum_{j\in\calJ_i}\Del_{j}(s,t)\mu_{j}=\la_i$ for all $i$.
By \eqref{41}, $\sum_{j\in \calJ^k}(T^n_{j}(t)-T^n_{j}(s))\les t-s$ for all $k$ and $s<t$.
Consequently the same holds for $T$. This shows
that $\Del(s,t)\in\slp$ for all $s<t$, a.s., namely $T$ has sample paths in
${\rm Lip}(\slp)$ a.s.

Finally, a process $\X$ is constructed.
By the above result, the sample paths of $T$ are a.e.\ differentiable,
with derivatives lying in $\slp$ a.e.
To construct the a.e.\ derivative as an $\{\calF'_t\}$-progressively measurable
process we adopt an argument from a proof in \cite{DM}.
Recall that the lexicographic order on $\R^{\calJ}$
is complete. That is, for every $A\subset\R^{\calJ}$ there exists
a (unique) greatest lower bound with respect to this order,
denoted here $\widetilde{\inf}A$. We can then construct
$$
\X_t=\lim_{\eps\downarrow0}\tildeinf_{s\in(t-\eps,t)\cap\Q}(t-s)^{-1}(T_t-T_s),
$$
where $\Q$ is the set of rationals.
Similarly to the proof in \cite[IV.17]{DM}, $\X$ is progressively measurable.
Moreover, one has a.s., for all $t$, $\int_0^t\X_sds=T_t$.
This gives $\X\in\frX(\{\calF'_t\})$.
\end{proof}

\begin{proof}[Proof of Lemma \ref{lem03}.]
We introduce notation special to this proof.
We fix $n$ and remove the superscript $n$ from the notation.
Write
$S_j$ for each of the $I+J$ processes originally denoted by $A^n_i$ and $S^n_{j}$.
Further, in place of $\bar T(t)$ (originally $\bar T^n(t)$)
we write $T(t)$.
There will be no confusion with the processes that are elsewhere denoted by
$S,T$, etc., because this proof regards the prelimit processes only.

The definition of an admissible control for the QCP
requires adaptedness of $T^n_t$ to $\calF^n_t$ (denoted here
as $T_t$ and $\calF_t$).
To ease the notation we first present the proof for simplified versions
of both $\calF_t$ and $\calG_u$ that do not include $\Ups$
as basis elements for the sigma algebras.
That is, for $\calF_t$ and $\calG_u$ we take
\[
\calF_t=\sig\{S(T_s):0\les s\les t\},
\qquad
t\in\R_+,
\]
\[
\calG_u=\sig\{S(v):0\LE v \LE u\},\qquad u\in\calU.
\]
Thus under the assumption that $T_t\in\calF_t$ for all $t$,
the aim is to show that $\{T_t\LE u\}\in\calG_u$ for all $t$ and $u$.
Only minor modifications are required to account for $\Ups$,
and these are described at the end of the proof.

Some further notation is as follows.
For $j\in \calJ^0$, let $\tau_j(0)=0$ and let the jump times of $S_j$
be denoted by $\tau_j(p)=\inf\{t:S_j(t)\ges p\}$, $p\in\N$.
By assumption, $\tau_j(0)<\tau_j(1)<\cdots$.
The vector $(\tau_j(1),\ldots,\tau_j(p))$ is denoted by
$\tau_j^{\leqslant p}$, where in the case $p=0$ this is set to the empty set.
For $N\in\Z_+^{\calJ^0}$, the collection $(\tau_j^{\leqslant N_j}: j\in \calJ^0)$ is
denoted by $\tau^{\LE N}$.
Next, let $\eta(0)=0$ and for $p\ges 1$, let
\[
\eta(p)=\inf\{t>\eta(p-1): \text{ for some $j\in\calJ^0$, } S_j(T_j(t))>S_j(T_j(\eta(p-1)))\}.
\]
For $p\in\Z_+$, let $N(p)$ denote the $\Z_+^{\calJ^0}$-valued RVs $N(p)=S(T(\eta(p)))$.
Then the sequence $0=\eta(0)<\eta(1)<\cdots$ are
the jump times of $\{S_j(T_j(t)),j\in\calJ^0\}$, and $N(p)$ the values
of $S(T(t))$ after each jump.
(Note that it is possible that at $\eta(p)$ multiple
trajectories $S_j(T_j(\cdot))$ jump simultaneously).
Also let $K_t=\max\{p:\eta(p)\les t\}$ denote the number of jumps by time~$t$.

It is not hard to see that, given $t$, the trajectory
$S(T(s))$, $s\in[0,t]$ can be recovered from $K_t$ and the collection
$\{(\eta(p\w K_t),N(p\w K_t),p\in\N\}$, and vice versa. As a result, we have the identity
\[
\calF_t=\sig\{K_t, \eta(p\w K_t), N(p\w K_t), p\in\N\}.
\]
Similarly,
for any $u\in\calU$, one can recover the trajectory $\{S(v):v\LE u\}$
from the collection $\{S(u),\tau_j(p\w S_j(u_j)),p\in\N,j\in\calJ^0\}$ and vice versa. Hence
we also have the identity
\begin{equation}\label{091}
	\calG_u=\sig\{S(u),\tau_j(p\w S_j(u_j)),p\in\N,j\in\calJ^0\}.
\end{equation}

We use Theorem I.18 in \cite{DM}, which states the following.
If $X$ is an RV on $(\Om,\calF)$
taking values in a measurable space $(E,\calE)$ and $Y$ an RV on $(\Om,\calF)$
with values in $(\calU,\calB_{\calU})$
which is $\sig(X)$-measurable then there exists a measurable map $\ph$
from $(E,\calE)$ to $(\calU,\calB_{\calU})$ such that $Y=\ph\circ X$.
Considering $E=\Z_+\times\R_+^\N\times(\Z_+^{\calJ^0})^\N$, endow $\Z_+$ and $\R_+$
with the Borel $\sig$-fields and let $\calE$ be the corresponding product $\sig$-field.
Then for every $t$ there exists a measurable map $\ph_t:
(E,\calE)\to(\calU,\calB_{\calU})$ such that
\[
T_t=\ph_t(K_t,\eta(p\w K_t), N(p\w K_t), p\in\N).
\]
For $\kappa\in\Z_+$ and $x=(\zeta,n)\in \R_+^\kappa\times(\Z_+^{\calJ^0})^\kappa$, denote
\[
\ph_t^{(\kappa)}(x)=\ph^{(\kappa)}_t(\zeta^{\les\kappa},n^{\les\kappa})
=\ph_t(\kappa,\zeta(p\w \kappa),n(p\w \kappa),p\in\N).
\]
Then we have
\begin{equation}\label{120}
	T_t=\ph_t^{(\kappa)}(\eta^{\les\kappa},N^{\les\kappa})
	\quad
	\text{holds on the event}
	\quad
	\{K_t=\kappa\}.
\end{equation}

Two properties of $\ph_t^{(\kappa)}$ can immediately be deduced
from properties of $T_t$. Namely,
the continuity and the monotonicity (w.r.t.\ $\LE$)
of the sample paths of $t\mapsto T_t$ imply that, for every $\kappa$ and $x$,
the map $t\mapsto\ph_t^{(\kappa)}(x)$ must also be continuous and monotone.
Another useful fact is as follows.
By the definitions of $\eta(p)$ and $\tau_j(p)$,
\begin{equation}\label{106}
	\eta(p+1)=\inf\{s>\eta(p): \text{ for some $j\in\calJ^0$,
		$T_j(s)\ges\tau_j(N_j(p)+1)$}\}.
\end{equation}
Hence using \eqref{120} and the aforementioned monotonicity,
\begin{equation}\label{102}
	\eta(p+1)=\inf\{s>\eta(p): \text{ for some $j\in \calJ^0$,
		$\ph_{j,s}^{(p)}(\eta^{\les p},N^{\les p})\ges\tau_j(N_j(p)+1)$}\}.
\end{equation}

Equipped with the maps $\ph^{(\kappa)}$ and these properties
we proceed with the following steps. Fix $u$ and $t$.

- Construct a process $\{\tilde T_s$, $s\in\R_+\}$
for which the jump times
are defined similarly to \eqref{102}, except that they
are based solely on the RVs $\tau^{\LE S(u)}$.
From its construction it will be immediate that
$\{\tilde T_t\LE u\}$ is $\calG_u$-measurable.

- Prove that $T_t\LE u$ iff $\tilde T_t\LE u$.

- Deduce that $\{T_t\LE u\}$ is $\calG_u$-measurable.

The construction of $\tilde T_s$ is as follows.
Recall that $t,u$ have been fixed and denote $\Sig=S(u)$.
Let
$\tilde\eta(0)=0$, $\tilde N(0)=0\in\Z_+^{\calJ^0}$. For $p\in\Z_+$ such that $\tilde\eta(p)<\iy$, define inductively
\begin{align*}
	G_p&=\{j\in \calJ^0:\tilde N_j(p)+1\les \Sig_j\},
	\\
	J^p(s)&=\{j\in G_p:\ph_{j,s}^{(p)}(\tilde\eta^{\les p},\tilde N^{\les p})
	\ges\tau_j(\tilde N_j(p)+1)\},
	\qquad s>\tilde\eta(p),
	\\
	\tilde\eta(p+1)&=\inf\{s>\tilde\eta(p): J^p(s)\ne\emptyset\}
	\\
	\tilde N(p+1)&=\tilde N(p)
	+\sum_{j\in J^p(\tilde\eta(p+1))}e_j
	\qquad \text{provided $\tilde\eta(p+1)<\iy$}.
\end{align*}
Note that after at most $|\Sig|$ steps, $G_p$ becomes empty, and as
a result $\tilde\eta(p_0+1)=\iy$ for some finite $p_0$. Based on the above construction,
define, for $0\les p\les p_0$,
\[
\tilde T_s=\ph_s^{(p)}(\tilde\eta^{\les p},\tilde N^{\les p}),
\qquad \text{for $s\in[\tilde\eta(p),\tilde\eta(p+1))$}.
\]
This defines $\tilde T_s$ for all $s\in\R_+$.

If instead of $G_p$ we had used $\calJ^0$ in the definition of $\tilde\eta(p+1)$,
we would have recovered $T_s$ itself as can be seen by \eqref{102}.
The use of $G_p$ achieves the goal that only the RVs $\Sig=S(u)$
and $\tau^{\LE\Sig}$ are involved in the construction.
As a result, the whole process $\{\tilde T_s, s\in\R_+\}$
is measurable on $\sig\{S(u),\tau^{\LE S(u)}\}$, hence by \eqref{091},
it is measurable on $\calG_u$. Consequently, the event $\{\tilde T_t\LE u\}$
is $\calG_u$-measurable.

The proof that $T_t\LE u$ iff $\tilde T_t\LE u$ proceeds in two steps.
First, assume
\begin{equation}\label{107}
	T_t\LE u.
\end{equation}
It is argued by induction that, for
$p\in\Z_+$,
\begin{equation}\label{101}
	\begin{split}
		&\text{if $\tilde\eta(p)\w\eta(p)\les t$ then}\\
		&
		\text{(a) $(\tilde\eta^{\les p},\tilde N^{\les p})=(\eta^{\les p},N^{\les p})$,}
		\\
		&\text{(b) for
			$s\in[\eta(p),\tilde\eta(p+1)\w\eta(p+1))$ one has $\tilde T_s=T_s$}.
	\end{split}
\end{equation}
In particular, \eqref{101} gives $\tilde T_s=T_s$,
$s\in[0,t]$, hence $\tilde T_t\LE u$.

Before proving \eqref{101}, notice that,
for every $k$, \eqref{101}(a) implies \eqref{101}(b),
because by \eqref{120} and the definition of $\tilde T_s$,
the processes $T_s$ on the interval $[\eta(p),\eta(p+1))$
and $\tilde T_s$ on the  interval $[\tilde\eta(p),\tilde\eta(p+1))$
are given by the same
formula, and since $\eta(p)=\tilde\eta(p)$, they must be equal for all
$s\in[\eta(p),\tilde\eta(p+1)\w\eta(p+1))$.
Therefore it suffices to prove \eqref{101}(a).
Also note that, by the definition of $\Sig$,
\begin{equation}\label{104}
	\tau_j(\Sig_j)\les u_j<\tau_j(\Sig_j+1),\qquad j\in\calJ^0.
\end{equation}

We now prove \eqref{101}(a).
First, for $p=0$, $(\tilde\eta(0),\tilde N(0))=(\eta(0),N(0))$ by definition.
Next, assume \eqref{101}(a) holds for $p$ and consider $p+1$.
Thus, under the condition
\[
\tilde\eta(p+1)\w\eta(p+1)\les t
\]
we must show
\begin{equation}\label{105}
	(\tilde\eta(p+1),\tilde N(p+1))=(\eta(p+1),N(p+1)).
\end{equation}
By the induction assumption,
$T_s=\tilde T_s=\ph_s^{(p)}(\tilde\eta^{\les p},\tilde N^{\les p})$
for $s\in[\eta(p),\tilde\eta(p+1)\w\eta(p+1))$.
Comparing \eqref{102} to the definition of $\tilde\eta(p+1)$
it is seen that $\eta(p+1)\les\tilde\eta(p+1)$
because $G_p\subset\calJ^0$. Thus
\begin{equation}\label{103}
	\eta(p+1)\les t.
\end{equation}
Arguing by contradiction, assume that \eqref{105} is false.
Then either (a) $\eta(p+1)<\tilde\eta(p+1)$ or
(b) $\eta(p+1)=\tilde\eta(p+1)$ but $N(p+1)\ne \tilde N(p+1)$.
In both cases,
by \eqref{106}, there exists $j\notin G_p$
such that one has $T_j(\eta(p+1))=\tau_j(N_j(p)+1)$.
Since the processes $\tilde T$ and $T$ are continuous and
equal on $[\eta(p),\eta(p+1))$, we have
\begin{equation}\label{108}
	\tilde T_j(\eta(p+1))=T_j(\eta(p+1))=\tau_j(N_j(p)+1).
\end{equation}
By the induction assumption, $\tilde N_j(p)=N_j(p)$,
hence $T_j(\eta(p+1))=\tau_j(\tilde N_j(p)+1)$.
Also, since $j\notin G_p$, $\tilde N_j(p)+1>\Sig_j$.
That is, $\tilde N_j(p)+1\ges\Sig_j+1$.
We therefore obtain
\begin{align*}
	\tau_j(\Sig_j+1)&\les
	\tau_j(\tilde N_j(p)+1)\\
	&=T_j(\eta(p+1)) &\text{by \eqref{108}}
	\\
	&\les T_j(t) &\text{by \eqref{103}}
	\\
	&\les u_j, &\text{by \eqref{107}}
	\\
	&<\tau_j(\Sig_j+1), &\text{by \eqref{104}}
\end{align*}
a contradiction. This shows that \eqref{101} holds under the assumption
$T_t\LE u$. Consequently, $T_t\LE u$ implies $\tilde T_t\LE u$.

Next assume
\begin{equation}\label{109}
	\tilde T_t\LE u.
\end{equation}
Under this condition we also prove that \eqref{101} holds.
The proof follows the same lines as above, except the last chain
of inequalities, that now reads
\begin{align*}
	\tau_j(\Sig_j+1)&\les
	\tau_j(\tilde N_j(p)+1)\\
	&=\tilde T_j(\eta(p+1)) &\text{by \eqref{108}}
	\\
	&\les \tilde T_j(t) &\text{by \eqref{103}}
	\\
	&\les u_j, &\text{by \eqref{109}}
	\\
	&<\tau_j(\Sig_j+1), &\text{by \eqref{104}}
\end{align*}
that again gives a contradiction, establishing \eqref{101}. As a result,
$\tilde T$ and $T$ are equal
on $[0,t]$, and thus $T_t\LE u$. This completes the proof that
$\tilde T_t\LE u$ iff $T_t\LE u$.
We deduce that $\{T_t\LE u\}$ is $\calG_u$-measurable.

Finally, accounting for the RV $\Ups$ by adding it to the definitions
of both $\calF_t$ and $\calG_u$, as in 
$\calF_t=\sig\{\Ups,S(T_s):0\les s\les t\}$, presents no difficulty.
Apart from applying this change,
augmenting the measurable space $(E,\calE)$ and updating the functions
$\ph^{(\kappa)}$ in an obvious way, the proof requires no adaptation.
\end{proof}

\begin{proof}[Proof of Lemma \ref{lem:bcp2wcp}.]
We keep the convention of writing $A_i$ as $S_{(i,0)}$.
However, unlike in the proof of the previous lemma,
$S^n$ (and similarly, $\bar T^n$, $\calG^n_u$, etc.)
and $S$ are distinct objects.

1. This part is a direct consequence of the fact that $\bar T_t$
is a $\calG_u$-stopping time for every $t$, and $\calH_t=\calG_{\bar T_t}$.
Indeed, to show that $\bar T_t$ is $\calH_t$-measurable, it suffices to show that
for every $v\in\calU$, $\{\bar T_t\LE v\}\in\calH_t$. But
$\calH_t=\{G\in\calG_0:G\cap\{\bar T_t\LE u\}\in\calG_u\}$
and $\{\bar T_t\LE v\}\cap\{\bar T_t\LE u\}=\{\bar T_t\LE v\w u\}$, where $v\w u$ denotes
componentwise minimum. Since the latter event is in $\calG_u$, statement 1
is proved.\vspace{0.3cm}

2. We first prove the independence
\begin{equation}\label{157}
	\calG_u \indep S(y)-S(x), \quad \text{for all } 0\LE u\LE x\LE y\in\calU.
\end{equation}
By the continuity of the sample paths of $S$,
it suffices to show this independence under the additional restriction
that for all $j\in \calJ^0$, $u_j<x_j$. Fix such $u,x,y$.
Denote $u_\eps=u+\eps\mathbf{e}$,  $\eps>0$.
Fix $\eps\in(0,1)$ so small that $u_{\eps}\LE x$.
Denote $\Del=S(y)-S(x)$ and $\Del^n=\hat S^n(y)-\hat S^n(x)$.

The proof of \eqref{157} is based on two facts.
First, one can find a sequence of RVs $\tilde\Del^n$
such that $\Del^n-\tilde\Del^n\to0$ in probability and at the same time,
for every $n$,
\begin{equation}\label{156}
	(\Ups,\{S^n(v):v\LE u_{\eps}\}) \indep \tilde\Del^n.
\end{equation}
Second, there exists a continuous approximation of the function
$\psi^{(u)}$ with certain properties.
More precisely, there exist three functions mapping
$\calU$ to itself, denoted
$\psi^{(u)}_{\rm cts}$, $g_1$, $g_2$,
where $\psi^{(u)}_{\rm cts}$ is continuous and $g_1$, $g_2$ are Borel,
such that
\begin{equation}\label{160}
	\psi^{(u)}_{\rm cts}=g_1\circ\psi^{(u_\eps)},
	\qquad
	\psi^{(u)}=g_2\circ\psi^{(u)}_{\rm cts}.
\end{equation}

We first show \eqref{157} based on these facts and then
prove these facts. To this end, let us show that,
for every $t$, the RV $\psi^{(u)}_{\rm cts}(\bar T^n_t)$ is measurable on $\calG^n_{u_{\eps}}$.
By the first part of \eqref{160}, it suffices to show that
$\psi^{(u_\eps)}(\bar T^n_t)$ is $\calG^n_{u_\eps}$-measurable.
To show this we use the following notation.
For $v\in\calU$ denote $R_v=\{w\in\calU:w\LE v\}$ and $R_v^c$ its complement in $\calU$.
Because the Borel $\sig$-field on $\calU$ is generated by
the rectangles $R_v$, $v\in\calU$, it suffices to show that
\[
\eta(n,t,u_\eps,v):=\{\psi^{(u_\eps)}(\bar T^n_t)\in R_v\}\in\calG^n_{u_\eps}
\]
for all $v\in\calU$.
By Lemma \ref{lem03}, for $v\in R_{\eps}$ one has
$\{\bar T^n_t\LE v\}\in\calG^n_v\subset\calG^n_{u_{\eps}}$,
where the fact that $\calG^n_\cdot$ forms a filtration is used.
Now, by the definition of $\psi^{(u_\eps)}$, we can write
\begin{align*}
	\eta(n,t,u_\eps,v)&=
	\begin{cases}
		\{\bar T^n_t\in R_v\} & \text{on } \{\bar T^n_t\in R_{u_\eps}\}
		\\
		\{u_\eps+\mathbf{e}\in R_v\} & \text{on } \{\bar T^n_t\in R_{u_\eps}\}^c
	\end{cases}
	\\
	&=\{\bar T^n_t\LE v\w u_\eps\}
	\cup [\{u_\eps+\mathbf{e}\LE v\}\cap\{\bar T^n_t\LE u_\eps\}^c].
\end{align*}
Since $v\w u_\eps\LE u_\eps$, the above expression shows that
$\eta(n,t,u_\eps,v)\in\calG^n_{u_\eps}$.
Hence follows the claim that $\psi^{(u)}_{\rm cts}(\bar T^n_t)$
is $\calG^n_{u_\eps}$-measurable.

Using this in \eqref{156} and recalling that $\hat S^n$ is a normalization
of $S^n$ gives
\begin{equation}\label{159}
	(\Ups,\{\hat S^n(v):v\LE u_{\eps}\},\{\psi^{(u)}_{\rm cts}(\bar T^n_t),t\ges0\})
	\indep \tilde\Del^n.
\end{equation}
We can now take limits. By Lemma \ref{lem11},
\[
(\{\hat S^n(v):v\LE u_{\eps}\},\{\bar T^n_t: t\ges0\},\Del^n)
\To(\{S(v):v\LE u_{\eps}\},\{\bar T_t: t\ges0\},\Del).
\]
Thus by the continuity of $\psi^{(u)}_{\rm cts}$
and the fact $\Del^n-\tilde\Del^n\to0$,
\[
(\{\hat S^n(v):v\LE u_{\eps}\},\{\psi^{(u)}_{\rm cts}(\bar T^n_t): t\ges0\},\tilde\Del^n)
\To(\{S(v):v\LE u_{\eps}\},\{\psi^{(u)}_{\rm cts}(\bar T_t): t\ges0\},\Del).
\]
By the independence stated in \eqref{159}, this shows
\[
(\{S(v):v\LE u_{\eps}\},\{\psi^{(u)}_{\rm cts}(\bar T_t): t\ges0\})
\indep \Del.
\]
We can now use the second part of \eqref{160} to obtain
\[
(\{S(v):v\LE u\},\{\psi^{(u)}(\bar T_t): t\ges0\}) \indep \Del.
\]
By the definition of $\calG_u$ in \eqref{155},
this shows that $\calG_u\indep\Del$, as stated in \eqref{157}.

To complete the proof of \eqref{157}, the two aforementioned facts
need to be established.
To construct $\tilde\Del^n$, let
$$
\tau_{j}^n\coloneqq \inf\left\lbrace z> v_{j},  S^n_{j}(z)> S^n_{j}(x_{j})\right\rbrace,
\qquad j\in\calJ^0.
$$
Let $(\tau^n)=(\tau^n_j)_{j\in\calJ^0}$ and $\tilde\Del^n=\hat S^n(\tau^n+y-x)-\hat S^n(\tau^n)$,
which, according to our convention, is a shorthand notation for
$\tilde\Del^n=(\tilde\Del^n_j)_{j\in\calJ^0}$,
$\tilde\Del^n_j=\hat S^n_j(\tau^n_j+y_j-x_j)-\hat S^n_j(\tau^n_j)$.
Since $\lambda_i^n$ and $\mu_{j}^n$ grow to $\infty$, it is easy to see
that $\tau^n_j\to x_i$ in probability.
Because $\hat S^n$ are $C$-tight, this implies the first desired property
of $\tilde\Del^n$, namely $\Del^n-\tilde\Del^n\to0$ in probability.
To prove the second property of $\tilde\Del^n$, \eqref{156}, due to the independence
of $\Ups$ and $S^n$, it suffices to show that
\begin{equation}\label{eq:indepincr}
	S^n(v) \indep \bar W^n:=S^n(\tau^n+y-x)-S^n(\tau^n),
\end{equation}
for every $v\LE u_\eps$. To this end we introduce some notation.
The accelerated interarrival and service times denoted in \eqref{e01} by $a^n_{i}(p)$ and $u^n_{j}(p)$
will now be denoted $u^n_j(p)$, $j\in\calJ^0$, $p=1,2,\ldots$, in accordance with the relabeling
convention used in this section. For $l\in \Z_+^{\calJ^0}$, introduce the multi-parameter filtration
\[
\bar\calF^n_{l}=\sig\{u^n_j(p_j),p_j\le l_j+1,j\in\calJ^0\}.
\]
Then it is easy to see that $\calT^n:=S^n(x)$ is
a (multi-parameter) stopping time w.r.t.\ $\bar\calF^n_\cdot$.
We now apply Lemma 7.6 of \cite{bw1}, which states that
the conditional distribution of the sequence
$(u^n_j(\calT^n_j+p), j\in\calJ^0)_{p=2,3,\ldots}$
given the stopped $\sig$-field $\bar\Sig^n:=\bar\calF^n_{\calT^n}$ is the same as the (unconditioned)
distribution of the original family of i.i.d.\ random variables
$(u^n_j(p), j\in\calJ^0)_{p=1,2,\ldots}$. (The lemma in \cite{bw1} and its proof are concerned
with the special case $|\calJ^0|=3$ but are valid in general).
Now, for each $j\in\calJ^0$, the RV $\bar W^n_j=S^n_j(\tau^n_j+y_j-x_j)-S^n_j(\tau^n_j)$
is determined by the sequence $(u^n_j(\calT^n_j+p), p=2,3,\ldots)$.
Therefore the conditional distribution of $\bar W^n$ given $\bar\Sig^n$ is the same as
the distribution of $S^n(y-x)$.
On the other hand, for each $v\LE x$, (in particular for each $v\LE u_\eps$),
the RV $S^n(v)$ is measurable w.r.t.\ $\bar\Sig^n$.
As a consequence, \eqref{eq:indepincr} holds.

Next we describe the construction of $\psi^{(u)}_{\rm cts}$.
For $v\in\calU$, set
\[
\psi^{(u)}_{\rm cts}(v)=(\psi^{(u)}_{{\rm cts},j}(v)),
\qquad \psi^{(u)}_{{\rm cts},j}(v)=(v_j+\eps^{-1}\max_g(v_g-u_g)^+)\w(u_j+1).
\]
Note that
\[
\psi^{(u)}_{\rm cts}=\psi^{(u)}_{\rm cts}\circ\psi^{(u_\eps)},
\]
because in $R_{u_\eps}$,
the function $\psi^{(u_\eps)}$ agrees with the identity map,
and if $v\in R^c_{u_\eps}$ then
$\psi^{(u)}_{\rm cts}(v)=u+\mathbf{e}$ whereas
\[
\psi^{(u)}_{\rm cts}\circ\psi^{(u_\eps)}(v)
=\psi^{(u)}_{\rm cts}(u_\eps+\mathbf{e})=u+\mathbf{e}.
\]
Also,
\[
\psi^{(u)}=\psi^{(u)}\circ\psi^{(u)}_{\rm cts}
\]
holds, because in $R_u$, both sides of the equality agree with
the identity map, and any member of $R^c_u$ is sent by
$\psi^{(u)}_{\rm cts}$ to a member of $R^c_u$,
resulting in $u+\mathbf{e}$ on both sides of the equation.
These two identities validate \eqref{160}.
This completes the proof of the first part of statement 2.

As for the second part of this statement, by the definition of $\calG_u$ in \eqref{155},
for any $t$, $\bar T_t$ is a $\mathcal{G}_u$ stopping time, and $\bar T_t$
has non-decreasing continuous sample paths a.s. Moreover, $S^\text{unit}$
is a SBM. Thus to show that the hypotheses of Theorem \ref{th:kurtz}
hold, it remains to show that for every subset $\calM\subset\calJ^0$, $\Ph^\calM_\theta(u)$ of \eqref{154},
with $S^\text{unit}$ in place of $\tilde S$, is a $\calG_u$-martingale.
To this end, note by \eqref{157} that for $u,v\in\calU$,
$S^\text{unit}(u+v)-S^\text{unit}(u)$ is independent of $\calG_u$. Hence, recalling $\calM\subset\calJ^0$ and 
$$
\Ph^{\calM}_\theta(u)=\prod\limits_{j\in\calM} \exp\left(\mathbf{i}\theta_{j}S^\text{\rm unit}_{j}(u_{j})
+\textstyle\frac{1}{2}\theta_{j}^2u_{j} \right),
$$
we have
\begin{align*}
	\E\left[\Ph^{\calM}_\theta(u+v)\mid \mathcal{G}_u\right]
	&=\Ph^{\calM}_\theta(u)\,\E\prod\limits_{j\in \calM} \exp\left\{\mathbf{i}\theta_{j} (S^\text{\rm unit}_{j}(u_{j}
	+v_{j})-S^\text{\rm unit}_{j}(u_{j}))+\textstyle\frac{1}{2}\theta_{j}^2v_{j} \right\}
	\\
	&=\Ph^{\calM}_\theta(u),
\end{align*}
where the last identity follows from the fact that $S^\text{\rm unit}_{j}(u_{j}
+v_{j})-S^\text{\rm unit}_{j}(u_{j})\sim\calN(0,{\rm diag}(v_j))$.
This shows that the tuple satisfies the hypotheses of Theorem \ref{th:kurtz}
and completes the proof of statement 2.\vspace{0.3cm}

3 and 4. We deduce from Theorem \ref{th:kurtz}
that $D_j=S_j\circ \bar T_j$, $j\in\calJ^0$ are continuous
local martingales w.r.t.\ $\calH_t$.
Moreover, for every distinct $j,j'\in\calJ^0$, $D_jD_{j'}$ is also
a continuous local martingale. Because the process $\lan D_j,D_{j'}\ran$,
by its definition, is the unique predictable process starting at zero which
makes $D_jD_{j'}-\lan D_j,D_{j'}\ran$ a local martingale, it follows that
$\lan D_j,D_{j'}\ran=0$.\vspace{0.3cm}

5. The proof is based on Proposition V.1.5(i) of \cite{rev-yor},
according to which if there exists a single-parameter filtration
with respect to which, for every $t$ and $j$, $\bar T_j(t)$ is a stopping time
and $S_j$ a martingale,
then $\langle S_{j}\circ\bar T_{j}\rangle_t=\langle S_{j}\rangle_{\bar T_{j}(t)}$.

Toward finding such a filtration, we argue as follows.
It suffices to obtain the result on a finite time interval $t\in[0,c_1]$, where
$c_1$ is arbitrary. Fix $c_1$ and let $c_2>c_1$. Fix $j^0$.
For $j\in\calJ^0$, let $e_j\in\calU$ denote the unit vector whose
$p$-th component is $\one_{p=j}$.
Consider the single-parameter filtration
$\calG^{(j^0)}_b$, $b\in[0,c_1]$ defined in terms of $\calG_u$ as
\[
\calG^{(j^0)}_b=\calG_{u(j^0,b)}
\qquad
u(j^0,b)=be_{j^0}+c_2\sum_{j\ne j^0}e_j,
\qquad
b\in[0,c_1].
\]
We use the fact that for all $t\in[0,c_1]$ and $j\in \calJ^0$, $\bar T_j(t)<c_2$.
Recall that $\bar T_t$ is a (multi-parameter) stopping time on
$\calG_u$. Thus for $t\in[0,c_1]$ and $b\in[0,c_1]$,
\[
\{\bar T_{j^0}(t)\le b\}=\{\bar T_{j^0}(t)\le b,\bar T_j(t)\le c_2 \text{ for all } j\ne j^0\}
=\{\bar T_t\LE u(j^0,b)\}\in\calG_{u(j^0,b)}=\calG^{(j^0)}_b.
\]
Hence $\bar T_{j^0}(t)$ is a $\{\calG^{({j^0})}_b\}$-stopping time for every $t\in[0,c_1]$.
Moreover, using \eqref{157} with $x=u=u(j^0,a)$ and $y=u(j^0,b)$, where
$0\le a\le b\le c_1$, shows that $\calG^{(j^0)}_a\indep S_{j^0}(b)-S_{j^0}(a)$, so
that the (zero drift) BM $S_{j^0}$ is a martingale on this filtration.
Clearly, $t\mapsto\bar T_{j^0}(t)$ is a.s.\ continuous non-decreasing. Hence
the assumptions of Proposition V.1.5(i) of \cite{rev-yor} are verified,
and so $\langle S_{j^0}\circ \bar T_{j^0}\rangle_t=\langle S_{j^0}\rangle_{\bar T_{j^0}(t)}=\sigma^2_{j^0}\bar T_{j^0}(t)$
for all $j^0\in\calJ^0$.
\end{proof}

%
%
\begin{appendix}
	\section{Some useful linear programming results}

For ease of reference we reproduce here several results related to linear programming from Chapter 7 of \cite{schrijver}.
We begin with the definitions of a convex polyhedron and a convex polytope in Section 7.2 of \cite{schrijver}.

\begin{definition}\label{def:con-poly}
(i) A set $P$ of vectors in $\R^{m}$ is called a convex polyhedron if $P=\left\lbrace x \mid Ax\leqslant b\right\rbrace$
for some matrix $A\in \R^{n\times m}$ and column vector $b\in \R^{n}$ .
\\
(ii) A set of vectors is a convex polytope if it is the convex hull of finitely many vectors.
\end{definition}

The following lemma is Corollary 7.1c in \cite{schrijver}.

\begin{lemma}\label{lem:polytope}
A set $P$ is a convex polytope if and only if $P$ is a bounded convex polyhedron.
\end{lemma}

For the remainder of this appendix we let
$A\in \R^{n\times m}$ be a matrix, $b\in \R^{n}$ be a column vector and $c\in \R^{m}$ be a row vector.  
We next state what is commonly referred to as the duality theorem of linear programming, which is Corollary 7.1g of \cite{schrijver}.

\begin{theorem}\label{th:duality}
We have 
\begin{equation}\label{eq:lpduality}
	\max_{x\in \R^m}\left\lbrace cx \mid Ax\leqslant b\right\rbrace=\min_{y\in \R^n}\left\lbrace yb \mid yA=c,\, y\geqslant 0\right\rbrace.
	\end{equation}
provided that both sets in \eqref{eq:lpduality} are non-empty.
\end{theorem}

Part (iii) of the next lemma is commonly referred to as complementary slackness, and is (35)(iii) in \cite{schrijver}.

\begin{lemma}\label{lem:comp-slack}
Assume that both optima in \eqref{eq:lpduality} are finite, and let $x_0, y_0$ be feasible solutions, respectively, of the maximization
and minimization problems (so that $Ax_0\leqslant b$,  $y_0A=c$, and $y_0 \geqslant 0$).
Then the following are equivalent:
\\
(i) $x_0$ and $y_0$ are optimal solutions in \eqref{eq:lpduality};
\\
(ii) $cx_0 =y_0b$;
\\
(iii) $y_0(b-Ax_0)=0$.
\end{lemma}

Finally, we close with strict complementary slackness, which is (36) in \cite{schrijver}.

\begin{theorem}\label{th:app}
Assume that both both sets in \eqref{eq:lpduality} are non-empty. 
For each $i=1,\ldots, n$, exactly one of the following  holds:
\begin{enumerate}[label=(\alph*)]
	\item\label{i} The maximum in \eqref{eq:lpduality} has a solution $x^0$ such that $(Ax^0)_i<b_i$
	\item\label{ii} The minimum in \eqref{eq:lpduality} has a solution $y^0$ such that $y^0_i>0$.
\end{enumerate}
\end{theorem}
\end{appendix}

\begin{funding}
	RA is supported by ISF grant 1035/20.
\end{funding}
\bibliographystyle{imsart-number}

\bibliography{refs}

\end{document}